\documentclass[12pt]{article}
\usepackage{booktabs}
\usepackage{caption}
\usepackage[fleqn]{amsmath}
\allowdisplaybreaks[4]
\usepackage{amsfonts,amsthm,amssymb,mathrsfs,bbding}
\usepackage{graphics,multicol}
\usepackage{graphicx}
\usepackage{color}
\usepackage{enumerate}
\usepackage{caption}
\usepackage{cite}
\usepackage{latexsym,bm}
\usepackage{mathtools}
\evensidemargin=\oddsidemargin\topmargin=-1.5cm

\newtheorem{thm}{Theorem}[section]

\newtheorem{lem}{Lemma}[section]

\newtheorem{pro}{Proposition}[section]

\newtheorem{remark}{Remark}

\theoremstyle{definition}
\newcommand{\tabincell}[2]{\begin{tabular}{@{}#1@{}}#2\end{tabular}}

\addtocounter{section}{0}

\begin{document}
\title{Graphs with at most three distance eigenvalues different from $-1$ and $-2$\footnote{This work is supported
by NSFC (Grant nos. 11671344 and 11531011).}}
\author{{\small Xueyi Huang, \ \ Qiongxiang Huang\footnote{
Corresponding author. E-mail: huangqx@xju.edu.cn.}, \ \ Lu Lu\setcounter{footnote}{-1}}\\[2mm]\scriptsize
College of Mathematics and Systems Science,
\scriptsize Xinjiang University, Urumqi, Xinjiang 830046, P. R. China}
\date{}
\maketitle
{\flushleft\large\bf Abstract}
Let $G$ be a connected graph on $n$ vertices, and let $D(G)$ be the distance matrix of $G$. Let $\partial_1(G)\ge\partial_2(G)\ge\cdots\ge\partial_n(G)$ denote the eigenvalues of $D(G)$. In this paper, we characterize all connected graphs with $\partial_{3}(G)\leq -1$ and $\partial_{n-1}(G)\geq -2$. By the way, we determine all connected graphs with at most three distance eigenvalues different from $-1$ and $-2$.
\vspace{0.1cm}
\begin{flushleft}
\end{flushleft}

\section{Introduction}\label{s-1}
Let $G$ be a connected simple graph with vertex set $V(G)=\{v_1,v_2,\ldots,v_n\}$.  Denote by $d_G(v_i,v_j)$ the length of the shortest path connecting $v_i$ and $v_j$ in $G$. The \emph{distance} between $v\in V(G)$ and $H$, a connected subgraph of $G$,  is defined to be  $d(v,H)=\min\{d_G(v,w)\mid w\in V(H)\}$.  Furthermore,  we define the \emph{diameter} and \emph{distance matrix} of $G$  as $d(G)=\max\{d_G(v_i,v_j)\mid v_i,v_j\in V(G)\}$ and $D(G)=[d_G(v_i,v_j)]_{n\times n}$, respectively. Then the characteristic polynomial $\Phi_G(x)=\det(xI-D(G))$ of $D(G)$ is also called the \emph{distance polynomial} ($D$-\emph{polynomial} for short) of $G$.  

Since $D(G)$ is a real and symmetric, its eigenvalues can be conveniently denoted and arranged as $\partial_1(G)\ge\partial_2(G)\ge\cdots\ge\partial_n(G)$.  These eigenvalues are also called the  \emph{distance eigenvalues} ($D$-\emph{eigenvalues} for short) of $G$. The \emph{distance spectrum} ($D$-\emph{spectrum} for short)  of $G$, denoted by $\mathrm{Spec}_D(G)$, is the multiset of $D$-eigenvalues of $G$. If $\alpha_1>\alpha_2>\cdots>\alpha_s$ denote all the distinct $D$-eigenvalues (with multiplicities $m_1,m_2,\ldots,m_s$, respectively) of $G$ , then the $D$-spectrum of $G$ can be written as $\mathrm{Spec}_D(G)=\{[\alpha_1]^{m_1},\ldots,[\alpha_s]^{m_s}\}$. Two connected graphs  are said to be \emph{distance cospectral} ($D$-\emph{cospectral} for short) if they share the same $D$-spectrum, and the graph $G$ is called \textit{determined by its $D$-spectrum} (\emph{DDS} for short)  if any connected graph distance cospectral with $G$ must be isomorphic to it. 

Throughout this paper, we denote by $G^c$ the \emph{complement} of $G$, $tG$ the disjoint union of $t$ copies of $G$, $N_G(v)$ the \emph{neighborhood} of $v\in V(G)$, $G[X]$ the induced subgraph of $G$ on $X\subseteq V(G)$, and $D_G(X)$ the principal submatrix of $D(G)$ corresponding to $G[X]$. Also, we denote by $P_n$ the path of order $n$, $K_n$ the complete graph on $n$ vertices, and $K_{n_1,\ldots,n_k}$ the complete $k$-partite graph with parts of order $n_1,\ldots,n_k$, respectively.

\begin{figure}[t]
\begin{center}
\unitlength 2mm 
\linethickness{0.4pt}
\ifx\plotpoint\undefined\newsavebox{\plotpoint}\fi 
\begin{picture}(20,8)(0,0)
\put(1,7){\line(0,-1){6}}
\put(1,7){\circle*{1.5}}
\put(1,1){\circle*{1.5}}
\put(7,7){\circle*{1.5}}
\put(7,1){\circle*{1.5}}
\put(13,7){\circle*{1.5}}
\put(13,1){\circle*{1.5}}
\put(19,7){\circle*{1.5}}
\put(19,1){\circle*{1.5}}
\put(19,7){\line(0,-1){6}}
\put(19,1){\line(-1,1){6}}
\put(13,1){\line(1,1){6}}
\put(13,7){\line(1,0){6}}
\put(13,1){\line(1,0){6}}
\put(13,1){\line(-1,0){12}}
\put(1,7){\line(1,0){12}}
\put(1,7){\line(1,-1){6}}
\put(1,1){\line(1,1){6}}
\put(7,7){\line(1,-1){6}}
\put(13,7){\line(-1,-1){6}}
\end{picture}
\caption{\small The graph $P_4[K_{2},K_{2}^c,K_{2}^c,K_{2}]$.}
\label{Figure-1}
\end{center}
\end{figure}
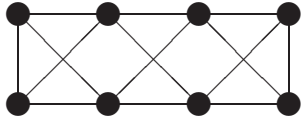

For a connected graph $G$ whose vertices are labeled as $v_1,v_2,\ldots,v_n$, and a sequence of graphs $H_1,H_2,\ldots,H_n$, the corresponding \emph{generalized lexicographic product} $G[H_1,\ldots,H_n]$ is defined as the graph obtained from $G$ by replacing  $v_i$ with the graph $H_i$ for $1\leq i\leq n$, and connecting all edges between $H_i$ and $H_j$ if $v_i$ is adjacent to $v_j$ for $1\leq i\neq j\leq n$. For example, Fig. \ref{Figure-1} illustrates the graph $P_4[K_{2},K_{2}^c,K_{2}^c,K_{2}]$.

Connected graphs whose $D$-eigenvalues possess special  properties arouse  some interests  in recent years. Lin et al. \cite{Lin2} (see also Yu \cite{Yu}) proved that $\partial_{n}(G)=-2$ if and only if $G$ is a complete multipartite graph, and conjectured that complete multipartite graphs are DDS.  Recently, Jin and Zhang \cite{Jin} confirmed the conjecture. Lin et al. \cite{Lin1,Lin} characterized all connected graphs with $\partial_n(G)\geq-1-\sqrt{2}$ and   $\partial_{n-1}(G)=-1$, respectively, and showed that these graphs are DDS. Li and Meng  \cite{Li} extended the result to connected graphs with $\partial_n(G)\geq-\frac{1+\sqrt{17}}{2}$. Xing and Zhou \cite{Xing} determined all connected graphs with $\partial_2(G)<-2+\sqrt{2}$, and Liu et al. \cite{Liu} generalized the result to $\partial_2(G)\leq\frac{17-\sqrt{329}}{2}$ and  proved that these graphs are DDS. Very recently, Lu et al. \cite{Lu} characterized all connected graphs with $\partial_3(G)\leq -1$ and $\partial_{n}(G)\geq -3$. It is worth noticing that most of the graphs mentioned above  are of diameter $2$.

On the other hand, in the past two decades, connected graphs with few distinct eigenvalues have been investigated for several graph matrices since such graphs always have pretty combinatorial properties. For some recent works on this topic, we refer the reader to \cite{Cheng,Cioaba,Cioaba1,Huang,Mohammadian,Rowlinson}. With regard to distance matrix, Koolen et al. \cite{Koolen} determined all connected graphs with three distinct $D$-eigenvalues of which two are simple; Lu et al. \cite{Lu} determined all connected graphs with exactly two $D$-eigenvalues different from $-1$ and $-3$ (which are also DDS);  Alazemi et al. \cite{Alazemi} characterized distance-regular graphs with diameter three having exactly three distinct $D$-eigenvalues, and also bipartite distance-regular graphs with diameter four having three distinct $D$-eigenvalues.

In this paper, we completely characterize the connected graphs with $\partial_{3}(G)\leq -1$ and $\partial_{n-1}(G)\geq -2$  (the diameter of these graphs could be $2$ or $3$). As a by-product, we also determine all connected graphs with at most three $D$-eigenvalues different from $-1$ and $-2$, which gives new classes of graphs with few distinct $D$-eigenvalues.

\section{Main tools}\label{s-2}
First of all, we present some results about the bounds of $\partial_{n}(G)$ and $\partial_{n-1}(G)$, which are useful in the subsequent sections.
\begin{lem}[\cite{Lin}]\label{Lemma-2-1}
Let $G$ be a connected graph on $n$ vertices. Then $\partial_{n}(G)\leq -d(G)$ where $d(G)$ is the diameter of $G$ and the equality holds if and only if $G$ is a complete multipartite graph.
\end{lem}

In particular, for graphs of diameter $2$, we have

\begin{lem}[\cite{Lin2}]\label{Lemma-2-2}
Let $G$ be a connected graph on $n$ vertices. Then $\partial_{n}(G)=-2$ with multiplicity $n-k$ if and only if $G$ is a complete $k$-partite graph for $2\leq k \leq n-1$.
\end{lem}

The following lemma determines all connected graphs with $\partial_{n-1}(G) \leq -1$.

\begin{lem}[\cite{Lin1}]\label{Lemma-2-3}
Let $G$ be a connected graph on $n$ vertices. If $n \geq 4$, then $\partial_{n-1}(G) \leq -1$ and the equality holds if and only if $G=K_r\vee(K_s\cup K_t)$  with $r\geq 1$.
\end{lem}

A \emph{Hermitian matrix} is a square matrix with complex entries that is equal to
its own conjugate transpose. Note that all the eigenvalues of a Hermitian matrix are real, and any real symmetric matrix is always a Hermitian matrix. The following  result is well known.

\begin{lem}[Cauchy Interlace Theorem]\label{Lemma-2-4}
Let $A$ be a Hermitian matrix of order $n$, and $B$ a principal submatrix of $A$ of order $m$. If $\lambda_1(A)\geq\lambda_2(A)\geq\cdots\geq\lambda_n(A)$ are the eigenvalues of $A$ and $\mu_1(B)\geq\mu_2(B)\geq\cdots\geq\mu_m(B)$ the eigenvalues of $B$, then $\lambda_i(A)\geq \mu_i(B)\geq\lambda_{n-m+i}(A)$ for $i=1,\ldots,m$.
\end{lem}

From Lemma \ref{Lemma-2-4} one can easily deduce the following result.
\begin{lem}\label{Lemma-2-5}
If $H$ is a connected induced subgraph of $G$ with diameter $d(H)<3$, then the $D$-eigenvalues of $H$ interlace those of $G$.
\end{lem}
Note that $\partial_{2}(K_{1,2})=1-\sqrt{3}$. By Lemma \ref{Lemma-2-5}, we have
\begin{lem}\label{Lemma-2-6}
Let $G$ be a connected graph with $n\geq 2$ vertices. Then $\partial_{2}(G)<1-\sqrt{3}$ if and only if $G$ is the complete graph $K_n$.
\end{lem}

Let $G$ be a connected graph on $n$ vertices, and let $S=\{v_1,\ldots,v_p\}\subseteq V(G)$ $(p\ge2)$ be a clique (resp. independent set) such that $N_G(v_i)\setminus S=N_G(v_j)\setminus S$ for $1\le i,j\le p$.  Take  $\mathbf{x}_\ell\in\mathbb{R}^n$ ($2\leq \ell\leq p$) as the vector defined on $V(G)$ with $\mathbf{x}_\ell(v_1)=1$, $\mathbf{x}_\ell(v_\ell)=-1$ and $\mathbf{x}_\ell(v)=0$ for $v\not\in\{v_1,v_\ell\}$, then one can easily verify that $D(G)\mathbf{x}_\ell=-\mathbf{x}_\ell$ (resp. $D(G)\mathbf{x}_\ell=-2\cdot\mathbf{x}_\ell$). Thus $-1$ (resp. $-2$) is a distance  eigenvalue of $G$ with multiplicity at least $p-1$ (cf. \cite{Lu}). If there are $r$ disjoint cliques (resp. independent sets) $S_1,\ldots,S_r$ ($|S_i|=p_i\geq 2$) of $V(G)$ sharing the same property as $S$, then we may conclude that $-1$ (resp. $-2$) is a distance  eigenvalue of $G$ with multiplicity at least $\sum_{i=1}^rp_i-r$. Thus we have the following result.

\begin{lem}\label{Lemma-2-7}
Let $G$ be a connected graph. If $S_1,\ldots,S_r$ ($|S_i|=p_i\geq 2$) are disjoint cliques (resp. independent sets) of $V(G)$ such that, for each $1\leq i\leq r$,  $N_G(u)\setminus S_i=N_G(v)\setminus S_i$ for any $u,v\in S_i$, then $-1$ (resp. $-2$) is a distance  eigenvalue of $G$ with multiplicity at least $\sum_{i=1}^rp_i-r$.
\end{lem}

For a connected  graph $G$ of order $n$, the vertex partition $\Pi:$ $V(G)=V_1\cup V_2\cup\cdots\cup V_k$ is called a \emph{distance equitable partition} if, for any $v\in V_i$, $\sum_{u\in V_j}d(v,u)=b_{ij}$ is a constant only dependent on $i,j$ ($1\le i,j\le k$). The matrix $B_\Pi=(b_{ij})_{k\times k}$ is called the \emph{distance divisor matrix} of $G$ with respect to $\Pi$.  The \emph{characteristic matrix} $\chi_\Pi$ of $\Pi$  is the $n\times k$ matrix whose columns are the character vectors of $V_1,\ldots,V_k$. 

The following lemma is an analogue of the result for adjacency matrix (cf. \cite{Godsil}, pp. 195--198), which states that the eigenvalues of $B_\Pi$ are also that of $D(G)$.
\begin{lem}\label{Lemma-2-8}
Let $G$ be a connected graph with distance matrix $D(G)$, and let $\Pi:V(G)=V_1\cup V_2\cup\cdots\cup V_k$ be a distance equitable partition of $G$ with distance divisor matrix $B_{\Pi}$. Then $\Psi_G(x)=\det(xI-B_{\Pi})|\Phi_G(x)=\det(xI-D(G))$, and the largest eigenvalue of $B_{\Pi}$ equals to  $\partial_1(G)$. In particular, the matrix $D(G)$ has the following two kinds of eigenvectors:
\begin{enumerate}[(i)]
\vspace{-0.2cm}
\item the eigenvectors in the column space of $\chi_\Pi$, and the corresponding eigenvalues coincide
with the eigenvalues of $B_\Pi$;
\vspace{-0.2cm}
\item the eigenvectors orthogonal to the columns of $\chi_\Pi$.
\end{enumerate}
\end{lem}

Let $G$ be a graph with vertex set $V(G)$. For any $X\subseteq V(G)$, we say that $X$ is \emph{$G$-connected} if the induced subgraph $G[X]$  is connected.
\begin{lem}[\cite{Seinsche}]\label{Lemma-2-9}
Let $G$ be a graph. The following statements are equivalent.
\begin{enumerate}[(i)]
\vspace{-0.2cm}
\item $G$ has no induced subgraph isomorphic to $P_4$.
\vspace{-0.2cm}
\item Every subset of $V(G)$ with more than one element is not $G$-connected or not $G^c$-connected.
\end{enumerate}
\end{lem}
Let $G_1$ and $G_2$ be two vertex disjoint graphs. The \emph{join} of $G_1$ and $G_2$, denoted by $G_1\vee G_2$, is the graph obtained from $G_1\cup G_2$ by connecting all edges between $G_1$ and $G_2$.  Let $G$ be a connected graph containing no induced $P_4$. Then $V(G)$ is a subset of itself and so is $G$-connected, by Lemma \ref{Lemma-2-9}, we know that $G^c$ is disconnected. Thus we obtain the following result.
\begin{lem}\label{Lemma-2-10}
If $G$ is a connected graph containing no induced $P_4$, then $G$ must be a join of two graphs, i.e., $G\cong G_1\vee G_2$, where $G_1$ and $G_2$ are non-null.
\end{lem}

\section{Graphs with $\partial_3(G)\leq -1$ and $\partial_{n-1}(G)\geq -2$}\label{s-3}
In this section, we focus on characterizing those graphs with $\partial_3(G)\leq -1$ and $\partial_{n-1}(G)\geq -2$. To achieve this goal, we need the following two crucial lemmas.
\begin{lem}\label{Lemma-3-1}
If $G$ is a connected graph on $n$ vertices with $\partial_3(G)\leq -1$ and $\partial_{n-1}(G)\geq -2$, then the graphs $F_1$--$F_7$ shown in Fig. \ref{Figure-2} cannot be induced subgraphs of $G$.
\end{lem}
\begin{proof}
By simple computation, it is seen that each $F_i$  ($1\leq i\leq 7$) has third largest $D$-eigenvalue greater than $-1$ or second least $D$-eigenvalue less than $-2$  (see Fig. \ref{Figure-2}). Then the result follows by Lemma \ref{Lemma-2-5} due to $d(F_i)=2$ for each $i$.
\end{proof}

\begin{figure}[t]
\begin{center}
\unitlength 1.5mm 
\linethickness{0.4pt}
\ifx\plotpoint\undefined\newsavebox{\plotpoint}\fi 
\begin{picture}(69,24)(0,-2)
\put(1,17){\line(1,0){12}}
\put(13,17){\line(-1,1){6}}
\put(7,23){\line(-1,-1){6}}
\put(18,17){\line(1,0){12}}
\put(30,17){\line(-1,1){6}}
\put(24,23){\line(-1,-1){6}}
\put(35,17){\line(1,0){12}}
\put(35,17){\line(1,1){6}}
\put(41,23){\line(1,-1){6}}
\put(52,17){\line(1,0){4}}
\put(64,17){\line(-1,1){6}}
\put(58,23){\line(-1,-1){6}}
\put(58,23){\line(-1,-3){2}}
\put(58,23){\line(1,-3){2}}
\put(60,17){\line(1,0){4}}
\put(58,23){\line(5,-3){10}}
\put(7,23){\circle*{1.5}}
\put(1,17){\circle*{1.5}}
\put(5,17){\circle*{1.5}}
\put(9,17){\circle*{1.5}}
\put(13,17){\circle*{1.5}}
\put(24,23){\circle*{1.5}}
\put(18,17){\circle*{1.5}}
\put(22,17){\circle*{1.5}}
\put(26,17){\circle*{1.5}}
\put(30,17){\circle*{1.5}}
\put(24,23){\line(-1,-3){2}}
\put(41,23){\line(-1,-3){2}}
\put(41,23){\line(1,-3){2}}
\put(41,23){\circle*{1.5}}
\put(35,17){\circle*{1.5}}
\put(39,17){\circle*{1.5}}
\put(43,17){\circle*{1.5}}
\put(47,17){\circle*{1.5}}
\put(58,23){\circle*{1.5}}
\put(52,17){\circle*{1.5}}
\put(56,17){\circle*{1.5}}
\put(60,17){\circle*{1.5}}
\put(64,17){\circle*{1.5}}
\put(68,17){\circle*{1.5}}
\put(7,14.5){\makebox(0,0)[cc]{\footnotesize$\partial_3=-0.3820$}}
\put(24,14.5){\makebox(0,0)[cc]{\footnotesize$\partial_3=-0.9125$}}
\put(41,14.5){\makebox(0,0)[cc]{\footnotesize$\partial_3=-0.7217$}}
\put(58,14.5){\makebox(0,0)[cc]{\footnotesize$\partial_5=-2.2223$}}
\put(7,12){\makebox(0,0)[cc]{\footnotesize$F_1$}}
\put(24,12){\makebox(0,0)[cc]{\footnotesize$F_2$}}
\put(41,12){\makebox(0,0)[cc]{\footnotesize$F_3$}}
\put(58,12){\makebox(0,0)[cc]{\footnotesize$F_4$}}
\put(1,9){\line(1,0){6}}
\put(7,9){\line(0,-1){6}}
\put(7,3){\line(-1,0){6}}
\put(1,3){\line(0,1){6}}
\put(1,9){\line(1,-1){6}}
\put(7,9){\line(-1,-1){6}}
\put(13,9){\line(0,-1){6}}
\put(13,3){\line(-2,1){12}}
\put(7,9){\line(1,-1){6}}
\put(13,9){\line(-2,-1){12}}
\put(13,9){\line(-1,-1){6}}
\put(20,3){\line(1,0){8}}
\put(28,3){\line(1,0){2}}
\put(20,3){\line(-1,0){2}}
\put(18,3){\line(2,3){4}}
\put(22,9){\line(1,-3){2}}
\put(22,9){\line(4,-3){8}}
\put(30,3){\line(-2,3){4}}
\put(26,9){\line(-1,-3){2}}
\put(26,9){\line(-4,-3){8}}
\put(35,3){\line(1,0){5}}
\put(40,3){\line(1,0){1}}
\put(41,3){\line(1,0){6}}
\put(35,3){\line(1,1){6}}
\put(41,9){\line(0,-1){6}}
\put(41,9){\line(1,-1){6}}
\put(41,9){\line(5,-3){10}}
\put(1,9){\circle*{1.5}}
\put(7,9){\circle*{1.5}}
\put(13,9){\circle*{1.5}}
\put(1,3){\circle*{1.5}}
\put(7,3){\circle*{1.5}}
\put(13,3){\circle*{1.5}}
\put(22,9){\circle*{1.5}}
\put(26,9){\circle*{1.5}}
\put(18,3){\circle*{1.5}}
\put(24,3){\circle*{1.5}}
\put(30,3){\circle*{1.5}}
\put(41,9){\circle*{1.5}}
\put(35,3){\circle*{1.5}}
\put(41,3){\circle*{1.5}}
\put(47,3){\circle*{1.5}}
\put(51,3){\circle*{1.5}}
\put(7,0.5){\makebox(0,0)[cc]{\footnotesize$\partial_5=-2.3589$}}
\put(24,0.5){\makebox(0,0)[cc]{\footnotesize$\partial_3=-0.8284$}}
\put(41,0.5){\makebox(0,0)[cc]{\footnotesize$\partial_3=-0.7667$}}
\put(7,-2){\makebox(0,0)[cc]{\footnotesize$F_5$}}
\put(24,-2){\makebox(0,0)[cc]{\footnotesize$F_6$}}
\put(41,-2){\makebox(0,0)[cc]{\footnotesize$F_7$}}
\end{picture}
\caption{\small The graphs $F_1$--$F_7$.}
\label{Figure-2}
\end{center}
\vspace{-0.3cm}
\end{figure}

\begin{lem}\label{Lemma-3-2}
If $G$ is a connected graph on $n$ ($n\geq 4$) vertices with $\partial_3(G)\leq -1$ and $\partial_{n-1}(G)\geq -2$, then each matrix listed below cannot be the principal submatrix of $D(G)$.
$$\begin{aligned}
&A_1\left[\begin{smallmatrix}
0&1&2&3&1\\
1&0&1&2&2\\
2&1&0&1&2\\
3&2&1&0&2\\
1&2&2&2&0
\end{smallmatrix}\right]~~~~~
A_2\left[\begin{smallmatrix}
0&1&2&3&1\\
1&0&1&2&2\\
2&1&0&1&2\\
3&2&1&0&3\\
1&2&2&3&0
\end{smallmatrix}\right]~~~~
A_3\left[\begin{smallmatrix}
0&1&2&3&1\\
1&0&1&2&2\\
2&1&0&1&3\\
3&2&1&0&2\\
1&2&3&2&0
\end{smallmatrix}\right]~~~~
A_4\left[\begin{smallmatrix}
0&1&2&3&1\\
1&0&1&2&2\\
2&1&0&1&3\\
3&2&1&0&3\\
1&2&3&3&0
\end{smallmatrix}\right]~~~~
A_5\left[\begin{smallmatrix}
0&1&2&3&2\\
1&0&1&2&1\\
2&1&0&1&1\\
3&2&1&0&2\\
2&1&1&2&0
\end{smallmatrix}\right]\\
&A_6\left[\begin{smallmatrix}
0&1&2&3&2&2\\
1&0&1&2&1&1\\
2&1&0&1&2&2\\
3&2&1&0&2&2\\
2&1&2&2&0&1\\
2&1&2&2&1&0
\end{smallmatrix}\right]~~~
A_7\left[\begin{smallmatrix}
0&1&2&3&2&2\\
1&0&1&2&1&1\\
2&1&0&1&2&2\\
3&2&1&0&2&3\\
2&1&2&2&0&1\\
2&1&2&3&1&0
\end{smallmatrix}\right]~~
A_8\left[\begin{smallmatrix}
0&1&2&3&2&2\\
1&0&1&2&1&1\\
2&1&0&1&2&2\\
3&2&1&0&3&2\\
2&1&2&3&0&1\\
2&1&2&2&1&0
\end{smallmatrix}\right]~~
A_9\left[\begin{smallmatrix}
0&1&2&3&2&2\\
1&0&1&2&1&1\\
2&1&0&1&2&2\\
3&2&1&0&3&3\\
2&1&2&3&0&1\\
2&1&2&3&1&0
\end{smallmatrix}\right]~
A_{10}\left[\begin{smallmatrix}
0&1&2&3&1&1\\
1&0&1&2&1&1\\
2&1&0&1&2&2\\
3&2&1&0&2&2\\
1&1&2&2&0&2\\
1&1&2&2&2&0
\end{smallmatrix}\right]\\
&A_{11}\left[\begin{smallmatrix}
0&1&2&3&1&1\\
1&0&1&2&1&1\\
2&1&0&1&2&2\\
3&2&1&0&2&3\\
1&1&2&2&0&2\\
1&1&2&3&2&0
\end{smallmatrix}\right]~
A_{12}\left[\begin{smallmatrix}
0&1&2&3&1&1\\
1&0&1&2&1&1\\
2&1&0&1&2&2\\
3&2&1&0&3&2\\
1&1&2&3&0&2\\
1&1&2&2&2&0
\end{smallmatrix}\right]~
A_{13}\left[\begin{smallmatrix}
0&1&2&3&1&1\\
1&0&1&2&1&1\\
2&1&0&1&2&2\\
3&2&1&0&3&3\\
1&1&2&3&0&2\\
1&1&2&3&2&0
\end{smallmatrix}\right]~
A_{14}\left[\begin{smallmatrix}
0&1&2&3&2&1\\
1&0&1&2&1&1\\
2&1&0&1&2&2\\
3&2&1&0&2&2\\
2&1&2&2&0&1\\
1&1&2&2&1&0
\end{smallmatrix}\right]~
A_{15}\left[\begin{smallmatrix}
0&1&2&3&2&1\\
1&0&1&2&1&1\\
2&1&0&1&2&2\\
3&2&1&0&2&2\\
2&1&2&2&0&2\\
1&1&2&2&2&0
\end{smallmatrix}\right]\\
&A_{16}\left[\begin{smallmatrix}
0&1&2&3&2&1\\
1&0&1&2&1&1\\
2&1&0&1&2&2\\
3&2&1&0&2&3\\
2&1&2&2&0&1\\
1&1&2&3&1&0
\end{smallmatrix}\right]~
A_{17}\left[\begin{smallmatrix}
0&1&2&3&2&1\\
1&0&1&2&1&1\\
2&1&0&1&2&2\\
3&2&1&0&2&3\\
2&1&2&2&0&2\\
1&1&2&3&2&0
\end{smallmatrix}\right]~
A_{18}\left[\begin{smallmatrix}
0&1&2&3&2&1\\
1&0&1&2&1&1\\
2&1&0&1&2&2\\
3&2&1&0&3&2\\
2&1&2&3&0&1\\
1&1&2&2&1&0
\end{smallmatrix}\right]~
A_{19}\left[\begin{smallmatrix}
0&1&2&3&2&1\\
1&0&1&2&1&1\\
2&1&0&1&2&2\\
3&2&1&0&3&2\\
2&1&2&3&0&2\\
1&1&2&2&2&0
\end{smallmatrix}\right]~
A_{20}\left[\begin{smallmatrix}
0&1&2&3&2&1\\
1&0&1&2&1&1\\
2&1&0&1&2&2\\
3&2&1&0&3&3\\
2&1&2&3&0&1\\
1&1&2&3&1&0
\end{smallmatrix}\right]\\
&A_{21}\left[\begin{smallmatrix}
0&1&2&3&2&1\\
1&0&1&2&1&1\\
2&1&0&1&2&2\\
3&2&1&0&3&3\\
2&1&2&3&0&2\\
1&1&2&3&2&0
\end{smallmatrix}\right]~
A_{22}\left[\begin{smallmatrix}
0&1&2&3&2&1\\
1&0&1&2&1&2\\
2&1&0&1&2&1\\
3&2&1&0&2&2\\
2&1&2&2&0&2\\
1&2&1&2&2&0
\end{smallmatrix}\right]~
A_{23}\left[\begin{smallmatrix}
0&1&2&3&2&1\\
1&0&1&2&1&2\\
2&1&0&1&2&1\\
3&2&1&0&2&2\\
2&1&2&2&0&3\\
1&2&1&2&3&0
\end{smallmatrix}\right]~
A_{24}\left[\begin{smallmatrix}
0&1&2&3&2&1\\
1&0&1&2&1&2\\
2&1&0&1&2&1\\
3&2&1&0&3&2\\
2&1&2&3&0&2\\
1&2&1&2&2&0
\end{smallmatrix}\right]~
A_{25}\left[\begin{smallmatrix}
0&1&2&3&2&1\\
1&0&1&2&1&2\\
2&1&0&1&2&1\\
3&2&1&0&3&2\\
2&1&2&3&0&3\\
1&2&1&2&3&0
\end{smallmatrix}\right]\\
&A_{26}\left[\begin{smallmatrix}
0&1&2&3&2&1\\
1&0&1&2&1&1\\
2&1&0&1&2&1\\
3&2&1&0&2&2\\
2&1&2&2&0&2\\
1&1&1&2&2&0
\end{smallmatrix}\right]~
A_{27}\left[\begin{smallmatrix}
0&1&2&3&2&1\\
1&0&1&2&1&1\\
2&1&0&1&2&1\\
3&2&1&0&3&2\\
2&1&2&3&0&2\\
1&1&1&2&2&0
\end{smallmatrix}\right]~
A_{28}\left[\begin{smallmatrix}
0&1&2&3&1&1\\
1&0&1&2&1&2\\
2&1&0&1&2&1\\
3&2&1&0&2&2\\
1&1&2&2&0&2\\
1&2&1&2&2&0
\end{smallmatrix}\right]~
A_{29}\left[\begin{smallmatrix}
0&1&2&3&1&1\\
1&0&1&2&1&2\\
2&1&0&1&2&1\\
3&2&1&0&3&2\\
1&1&2&3&0&2\\
1&2&1&2&2&0
\end{smallmatrix}\right]~
A_{30}\left[\begin{smallmatrix}
0&1&2&3&1&1\\
1&0&1&2&1&1\\
2&1&0&1&2&1\\
3&2&1&0&2&2\\
1&1&2&2&0&2\\
1&1&1&2&2&0
\end{smallmatrix}\right]\\
&A_{31}\left[\begin{smallmatrix}
0&1&2&3&1&1\\
1&0&1&2&1&1\\
2&1&0&1&2&1\\
3&2&1&0&3&2\\
1&1&2&3&0&2\\
1&1&1&2&2&0
\end{smallmatrix}\right]~
A_{32}\left[\begin{smallmatrix}
0&1&2&3&2&2\\
1&0&1&2&1&1\\
2&1&0&1&2&2\\
3&2&1&0&2&1\\
2&1&2&2&0&1\\
2&1&2&1&1&0
\end{smallmatrix}\right]~
A_{33}\left[\begin{smallmatrix}
0&1&2&3&2&2\\
1&0&1&2&1&1\\
2&1&0&1&2&1\\
3&2&1&0&2&1\\
2&1&2&2&0&1\\
2&1&1&1&1&0
\end{smallmatrix}\right]~
A_{34}\left[\begin{smallmatrix}
0&1&2&3&1&2\\
1&0&1&2&1&1\\
2&1&0&1&2&2\\
3&2&1&0&2&1\\
1&1&2&2&0&1\\
2&1&2&1&1&0
\end{smallmatrix}\right]~
A_{35}\left[\begin{smallmatrix}
0&1&2&3&1&2\\
1&0&1&2&1&1\\
2&1&0&1&2&1\\
3&2&1&0&2&1\\
1&1&2&2&0&1\\
2&1&1&1&1&0
\end{smallmatrix}\right]\\
&A_{36}\left[\begin{smallmatrix}
0&1&2&3&2&2\\
1&0&1&2&1&2\\
2&1&0&1&2&1\\
3&2&1&0&2&2\\
2&1&2&2&0&1\\
2&2&1&2&1&0
\end{smallmatrix}\right]~
A_{37}\left[\begin{smallmatrix}
0&1&2&3&2&3\\
1&0&1&2&1&2\\
2&1&0&1&2&1\\
3&2&1&0&2&2\\
2&1&2&2&0&1\\
3&2&1&2&1&0
\end{smallmatrix}\right]~
A_{38}\left[\begin{smallmatrix}
0&1&2&3&2&2\\
1&0&1&2&1&2\\
2&1&0&1&2&1\\
3&2&1&0&3&2\\
2&1&2&3&0&1\\
2&2&1&2&1&0
\end{smallmatrix}\right]~
A_{39}\left[\begin{smallmatrix}
0&1&2&3&2&3\\
1&0&1&2&1&2\\
2&1&0&1&2&1\\
3&2&1&0&3&2\\
2&1&2&3&0&1\\
3&2&1&2&1&0
\end{smallmatrix}\right]~
A_{40}\left[\begin{smallmatrix}
0&1&2&3&2&2\\
1&0&1&2&1&2\\
2&1&0&1&2&1\\
3&2&1&0&2&1\\
2&1&2&2&0&1\\
2&2&1&1&1&0
\end{smallmatrix}\right]\\
&A_{41}\left[\begin{smallmatrix}
0&1&2&3&2&3\\
1&0&1&2&1&2\\
2&1&0&1&2&1\\
3&2&1&0&2&1\\
2&1&2&2&0&1\\
3&2&1&1&1&0
\end{smallmatrix}\right]~
A_{42}\left[\begin{smallmatrix}
0&1&2&3&1&2\\
1&0&1&2&1&2\\
2&1&0&1&2&1\\
3&2&1&0&2&1\\
1&1&2&2&0&1\\
2&2&1&1&1&0
\end{smallmatrix}\right]~
A_{43}\left[\begin{smallmatrix}
0&1&2&3&1&1\\
1&0&1&2&2&2\\
2&1&0&1&1&1\\
3&2&1&0&2&2\\
1&2&1&2&0&1\\
1&2&1&2&1&0
\end{smallmatrix}\right]~
A_{44}\left[\begin{smallmatrix}
0&1&2&3&1&1\\
1&0&1&2&1&1\\
2&1&0&1&1&1\\
3&2&1&0&2&2\\
1&1&1&2&0&2\\
1&1&1&2&2&0
\end{smallmatrix}\right]~
A_{45}\left[\begin{smallmatrix}
0&1&2&3&1&1\\
1&0&1&2&2&1\\
2&1&0&1&1&1\\
3&2&1&0&2&2\\
1&2&1&2&0&1\\
1&1&1&2&1&0
\end{smallmatrix}\right]\\
&A_{46}\left[\begin{smallmatrix}
0&1&2&3&1&1\\
1&0&1&2&2&1\\
2&1&0&1&1&1\\
3&2&1&0&2&2\\
1&2&1&2&0&2\\
1&1&1&2&2&0
\end{smallmatrix}\right]~
A_{47}\left[\begin{smallmatrix}
0&1&2&3&1&2\\
1&0&1&2&2&1\\
2&1&0&1&1&2\\
3&2&1&0&2&1\\
1&2&1&2&0&2\\
2&1&2&1&2&0
\end{smallmatrix}\right]~
A_{48}\left[\begin{smallmatrix}
0&1&2&3&1&2\\
1&0&1&2&2&1\\
2&1&0&1&1&2\\
3&2&1&0&2&1\\
1&2&1&2&0&3\\
2&1&2&1&3&0
\end{smallmatrix}\right]~
A_{49}\left[\begin{smallmatrix}
0&1&2&3&1&2\\
1&0&1&2&2&1\\
2&1&0&1&1&1\\
3&2&1&0&2&1\\
1&2&1&2&0&2\\
2&1&1&1&2&0
\end{smallmatrix}\right]~
A_{50}\left[\begin{smallmatrix}
0&1&2&3&1&2\\
1&0&1&2&1&1\\
2&1&0&1&1&1\\
3&2&1&0&2&1\\
1&1&1&2&0&2\\
2&1&1&1&2&0
\end{smallmatrix}\right]\\
&A_{51}\left[\begin{smallmatrix}
0&1&2&3&1&2\\
1&0&1&2&2&1\\
2&1&0&1&1&2\\
3&2&1&0&2&1\\
1&2&1&2&0&1\\
2&1&2&1&1&0
\end{smallmatrix}\right]
\end{aligned}
$$
\end{lem}
\begin{proof}
According to Tab. \ref{Table-1}, each $A_i$ ($1\leq i\leq 51$) has third largest  eigenvalue greater than $-1$ or second least eigenvalue less than $-2$. Thus our result follows by Lemma \ref{Lemma-2-4}. 
\end{proof}
\begin{table}[t]
\small
\caption{The third largest or second least eigenvalues of $A_1$--$A_{51}$.}
\vspace{-0.5cm}
\begin{center}
\scriptsize
\begin{tabular}{cccccccc}
\hline
$A_i$&$\partial_3$ or $\partial_5$&$A_i$&$\partial_3$ or $\partial_5$&$A_i$&$\partial_3$ or $\partial_5$&$A_i$&$\partial_3$ or $\partial_5$\\
\hline
$A_1$&$\partial_3=-0.6557$&$A_2$&$\partial_3=-0.9321$&$A_3$&$\partial_3=-0.6286$&$A_4$&$\partial_3=-0.6012$\\
$A_5$&$\partial_3=-0.8365$&$A_6$&$\partial_5=-2.5294$&$A_7$&$\partial_5=-2.4413$&$A_8$&$\partial_5=-2.4413$\\
$A_9$&$\partial_5=-2.3224$&$A_{10}$&$\partial_3=-0.7666$&$A_{11}$&$\partial_3=-0.7520$&$A_{12}$&$\partial_3=-2.0671$\\
$A_{13}$&$\partial_3=-0.6851$&$A_{14}$&$\partial_5=-2.1099$&$A_{15}$&$\partial_5=-2.1725$&$A_{16}$&$\partial_5=-2.1099$\\
$A_{17}$&$\partial_5=-2.0898$&$A_{18}$&$\partial_3=-0.5714$&$A_{19}$&$\partial_5=-2.4413$&$A_{20}$&$\partial_3=-0.6851$\\
$A_{21}$&$\partial_5=-2.3224$&$A_{22}$&$\partial_5=-2.6712$&$A_{23}$&$\partial_5=-3.4142$&$A_{24}$&$\partial_5=-2.5829$\\
$A_{25}$&$\partial_5=-3.1708$&$A_{26}$&$\partial_5=-2.4216$&$A_{27}$&$\partial_5=-2.3862$&$A_{28}$&$\partial_3=-0.4353$\\
$A_{29}$&$\partial_3=-0.8401$&$A_{30}$&$\partial_3=-0.4523$&$A_{31}$&$\partial_3=-0.6010$&$A_{32}$&$\partial_3=-0.8303$\\
$A_{33}$&$\partial_3=-0.6712$&$A_{34}$&$\partial_3=-0.6535$&$A_{35}$&$\partial_3=-0.4679$&$A_{36}$&$\partial_5=-2.3391$\\
$A_{37}$&$\partial_5=-2.5829$&$A_{38}$&$\partial_5=-2.5829$&$A_{39}$&$\partial_5=-3.1708$&$A_{40}$&$\partial_3=-0.8636$\\
$A_{41}$&$\partial_3=-0.8401$&$A_{42}$&$\partial_3=-0.7720$&$A_{43}$&$\partial_5=-2.3770$&$A_{44}$&$\partial_3=-0.7465$\\
$A_{45}$&$\partial_3=-0.7465$&$A_{46}$&$\partial_5=-2.3770$&$A_{47}$&$\partial_3=-0.4607$&$A_{48}$&$\partial_3=0$\\
$A_{49}$&$\partial_3=-0.6535$&$A_{50}$&$\partial_3=-0.4679$&$A_{51}$&$\partial_3=-0.7720$&--&--\\
\hline
\end{tabular}
\end{center}
\label{Table-1}
\end{table}

Now we begin to prove the main result of this section.
\begin{pro}\label{Pro-3-1}
Let  $G$ be a connected graph on $n$ vertices with $\partial_3(G)\leq -1$ and $\partial_{n-1}(G)\geq -2$. Then one of the following occurs:
\begin{enumerate}[(1)]
\vspace{-0.2cm}
\item $d(G)\leq 2$ and $G\in\{I_i\mid 1\leq i\leq 7\}$, where $I_1=K_n$ ($n\geq 4$), $I_2=K_a\vee K_b^c$ ($a\geq 2,b\geq 1$), $I_3=K_a\vee (K_b\cup K_c)$ ($a,b,c\geq 2$), $I_4=K_a\vee (K_b\cup K_c^c)$ ($a,b\geq 2$, $c\geq 1$), $I_5=K_a^c\vee K_{b}^c$ ($a,b\geq 1$), $I_6=K_a^c\vee (K_b\cup K_{c})$ ($a\geq 1$, $b,c\geq 2$) and $I_7=K_a^c\vee (K_b\cup K_{c}^c)$ ($a,c\geq 1$, $b\geq 2$);
\vspace{-0.2cm}

\item $d(G)=3$ and $G\in\{J_i\mid 1\leq i\leq 8\}$, where $J_1=P_4[K_{a}^c,K_{b},K_{c}^c,K_{d}^c]$, $J_2=P_4[K_{a}^c,K_{b},K_{c},K_{d}^c]$, $J_3=P_4[K_{a}^c,K_{b},K_{c}^c,K_{d}]$, $J_4=P_4[K_{a}^c,K_{b}^c,K_{c},K_{d}]$, $J_5=P_4[K_{a}^c,K_{b},K_{c},K_{d}]$, $J_6=P_4[K_{a},K_{b},K_{c}^c,K_{d}]$ and $J_7=P_4[K_{a},K_{b},K_{c},K_{d}]$, where $a,b,c,d\geq 1$.
\end{enumerate}
\end{pro}
\begin{proof}
Let $d(G)$ be the diameter of $G$. If $d(G)\geq 4$, then $D(P_5)$ is a principal submatrix of  $D(G)$,  and so  
$-0.7639=\partial_3(P_5)\leq\partial_{3}(G)= -1$ by Lemma \ref{Lemma-2-4}, which is impossible. Now we consider the following two cases.

\vspace{0.2cm}

\noindent\textbf{Case 1.} $d(G)\leq 2$.

First of all, we prove that $G$ cannot contain $P_4$ as its induced subgraph. Suppose to the contrary that $P_4=v_1v_2v_3v_4$ is an induce subgraph of $G$. Then there exists some vertex $v\in V(G)$  which is adjacent to both $v_1$ and $v_4$ because $d_G(u,v)=2$ due to $d(G)\leq 2$. Thus at least one of $\{F_{1},F_{2},F_{3}\}$  is  the induce subgraph of $G$, which is impossible by Lemma \ref{Lemma-3-1}. Therefore, from Lemma \ref{Lemma-2-10}, the exist two non-null graphs $G_1$ and $G_2$ such that $G=G_1\vee G_2$. We only need to discuss the following two situations.

\vspace{0.2cm}

\noindent\textbf{Subcase 1.1.} Both $G_1$ and $G_2$ contain no induced $P_3$. 

Since $P_3$ is not an induced subgraph of $G_1$ and $G_2$, we claim that both $G_1$ and $G_2$ are the disjoint unions of some complete graphs. Further, if $G_1$ or $G_2$ contains $2K_2\cup K_1$ as its induced subgraph, then $G=G_1\vee G_2$ contains induced $F_4$, which is a contradiction by Lemma \ref{Lemma-3-1}. Thus, for $i=1,2$, we  conclude that $G_i$ is  one of the following graphs: $K_a$ ($a\geq 2$), $K_a^{c}$ ($a\geq 1$), $K_a\cup K_b$ ($a,b\geq 2$) and $K_a\cup K_b^c$ ($a\geq 2$, $b\geq 1$). Therefore, all the possible forms of $G$ are $I_1=K_n$ ($n\geq 4$), $I_2=K_a\vee K_b^c$ ($a\geq 2,b\geq 1$), $I_3=K_a\vee (K_b\cup K_c)$ ($a,b,c\geq 2$), $I_4=K_a\vee (K_b\cup K_c^c)$ ($a,b\geq 2$, $c\geq 1$), $I_5=K_a^c\vee K_{b}^c$ ($a,b\geq 1$), $I_6=K_a^c\vee (K_b\cup K_{c})$ ($a\geq 1$, $b,c\geq 2$), $I_7=K_a^c\vee (K_b\cup K_{c}^c)$ ($a,c\geq 1$, $b\geq 2$), $I_8=(K_{a}\cup K_{b})\vee (K_{c}\cup K_{d})$ ($a,b,c,d\geq 2$), $I_9=(K_{a}\cup K_{b})\vee (K_{c}\cup K_{d}^c)$ ($a,b,c\geq 2$, $d\geq 1$) and $I_{10}=(K_{a}\cup K_{b}^c)\vee (K_{c}\cup K_{d}^c)$ ($a,c\geq 2$, $b,d\geq 1$). By Lemma \ref{Lemma-3-1}, we know that $F_{5}$ cannot be an induced subgraph of $G$, which implies that  $G\not\in\{I_8,I_9,I_{10}\}$. Hence, we may conclude that $G\in\{I_i\mid 1\leq i\leq 7\}$ in this situation.
\vspace{0.2cm}

\noindent\textbf{Subcase 1.2.}  At least one of $G_1$ and $G_2$ contains induced $P_3$. 

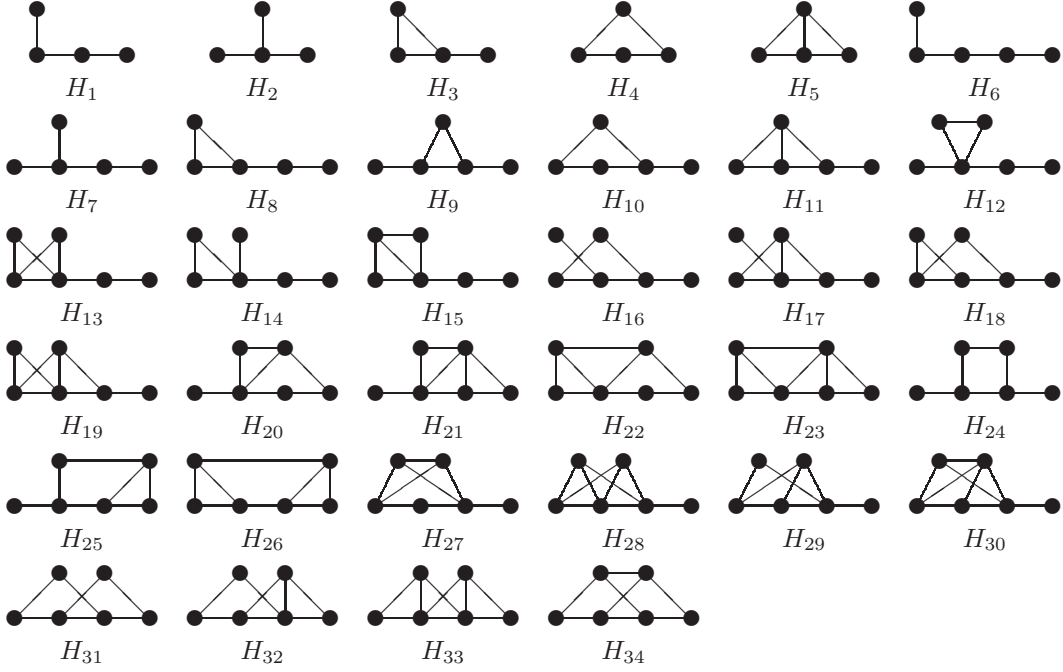
\begin{figure}[t]
\begin{center}
\unitlength 1.5mm 
\linethickness{0.4pt}
\ifx\plotpoint\undefined\newsavebox{\plotpoint}\fi 
\begin{picture}(94,56)(0,-2)
\put(3,51){\line(0,1){4}}
\put(3,51){\line(1,0){8}}
\put(19,51){\line(1,0){8}}
\put(35,51){\line(1,0){8}}
\put(51,51){\line(1,0){8}}
\put(67,51){\line(1,0){8}}
\put(23,51){\line(0,1){4}}
\put(35,51){\line(0,1){4}}
\put(35,55){\line(1,-1){4}}
\put(51,51){\line(1,1){4}}
\put(55,55){\line(1,-1){4}}
\put(67,51){\line(1,1){4}}
\put(71,55){\line(1,-1){4}}
\put(71,55){\line(0,-1){4}}
\put(81,51){\line(1,0){12}}
\put(81,55){\line(0,-1){4}}
\put(3,55){\circle*{1.5}}
\put(3,51){\circle*{1.5}}
\put(7,51){\circle*{1.5}}
\put(11,51){\circle*{1.5}}
\put(23,55){\circle*{1.5}}
\put(19,51){\circle*{1.5}}
\put(23,51){\circle*{1.5}}
\put(27,51){\circle*{1.5}}
\put(35,55){\circle*{1.5}}
\put(35,51){\circle*{1.5}}
\put(39,51){\circle*{1.5}}
\put(43,51){\circle*{1.5}}
\put(55,55){\circle*{1.5}}
\put(51,51){\circle*{1.5}}
\put(55,51){\circle*{1.5}}
\put(59,51){\circle*{1.5}}
\put(71,55){\circle*{1.5}}
\put(67,51){\circle*{1.5}}
\put(71,51){\circle*{1.5}}
\put(75,51){\circle*{1.5}}
\put(81,55){\circle*{1.5}}
\put(81,51){\circle*{1.5}}
\put(85,51){\circle*{1.5}}
\put(89,51){\circle*{1.5}}
\put(93,51){\circle*{1.5}}
\put(7,48){\makebox(0,0)[cc]{\footnotesize$H_1$}}
\put(23,48){\makebox(0,0)[cc]{\footnotesize$H_2$}}
\put(39,48){\makebox(0,0)[cc]{\footnotesize$H_3$}}
\put(55,48){\makebox(0,0)[cc]{\footnotesize$H_4$}}
\put(71,48){\makebox(0,0)[cc]{\footnotesize$H_5$}}
\put(87,48){\makebox(0,0)[cc]{\footnotesize$H_6$}}
\put(1,41){\line(1,0){12}}
\put(17,41){\line(1,0){12}}
\put(33,41){\line(1,0){12}}
\put(49,41){\line(1,0){12}}
\put(65,41){\line(1,0){12}}
\put(81,41){\line(1,0){12}}
\put(5,45){\line(0,-1){4}}
\put(17,45){\line(0,-1){4}}
\put(17,45){\line(1,-1){4}}
\multiput(37,41)(.03333333,.06666667){60}{\line(0,1){.06666667}}
\multiput(39,45)(.03333333,-.06666667){60}{\line(0,-1){.06666667}}
\put(49,41){\line(1,1){4}}
\put(53,45){\line(1,-1){4}}
\put(65,41){\line(1,1){4}}
\put(69,45){\line(0,-1){4}}
\put(69,45){\line(1,-1){4}}
\multiput(85,41)(-.03333333,.06666667){60}{\line(0,1){.06666667}}
\put(83,45){\line(1,0){4}}
\multiput(87,45)(-.03333333,-.06666667){60}{\line(0,-1){.06666667}}
\put(1,41){\circle*{1.5}}
\put(5,41){\circle*{1.5}}
\put(5,45){\circle*{1.5}}
\put(9,41){\circle*{1.5}}
\put(13,41){\circle*{1.5}}
\put(17,45){\circle*{1.5}}
\put(17,41){\circle*{1.5}}
\put(21,41){\circle*{1.5}}
\put(25,41){\circle*{1.5}}
\put(29,41){\circle*{1.5}}
\put(39,45){\circle*{1.5}}
\put(37,41){\circle*{1.5}}
\put(41,41){\circle*{1.5}}
\put(33,41){\circle*{1.5}}
\put(45,41){\circle*{1.5}}
\put(49,41){\circle*{1.5}}
\put(53,41){\circle*{1.5}}
\put(57,41){\circle*{1.5}}
\put(61,41){\circle*{1.5}}
\put(53,45){\circle*{1.5}}
\put(69,45){\circle*{1.5}}
\put(65,41){\circle*{1.5}}
\put(69,41){\circle*{1.5}}
\put(73,41){\circle*{1.5}}
\put(77,41){\circle*{1.5}}
\put(83,45){\circle*{1.5}}
\put(87,45){\circle*{1.5}}
\put(85,41){\circle*{1.5}}
\put(81,41){\circle*{1.5}}
\put(89,41){\circle*{1.5}}
\put(93,41){\circle*{1.5}}
\put(7,38){\makebox(0,0)[cc]{\footnotesize$H_7$}}
\put(23,38){\makebox(0,0)[cc]{\footnotesize$H_8$}}
\put(39,38){\makebox(0,0)[cc]{\footnotesize$H_9$}}
\put(55,38){\makebox(0,0)[cc]{\footnotesize$H_{10}$}}
\put(71,38){\makebox(0,0)[cc]{\footnotesize$H_{11}$}}
\put(87,38){\makebox(0,0)[cc]{\footnotesize$H_{12}$}}
\put(1,31){\line(1,0){12}}
\put(17,31){\line(1,0){12}}
\put(33,31){\line(1,0){12}}
\put(49,31){\line(1,0){12}}
\put(65,31){\line(1,0){12}}
\put(81,31){\line(1,0){12}}
\put(1,31){\line(0,1){4}}
\put(1,35){\line(1,-1){4}}
\put(5,31){\line(0,1){4}}
\put(5,35){\line(-1,-1){4}}
\put(17,31){\line(0,1){4}}
\put(17,35){\line(1,-1){4}}
\put(21,31){\line(0,1){4}}
\put(33,31){\line(0,1){4}}
\put(33,35){\line(1,-1){4}}
\put(37,31){\line(0,1){4}}
\put(37,35){\line(-1,0){4}}
\put(49,31){\line(1,1){4}}
\put(53,35){\line(1,-1){4}}
\put(53,31){\line(-1,1){4}}
\put(65,31){\line(1,1){4}}
\put(69,35){\line(1,-1){4}}
\put(69,35){\line(0,-1){4}}
\put(69,31){\line(-1,1){4}}
\put(81,35){\line(0,-1){4}}
\put(81,35){\line(1,-1){4}}
\put(81,31){\line(1,1){4}}
\put(85,35){\line(1,-1){4}}
\put(1,35){\circle*{1.5}}
\put(5,35){\circle*{1.5}}
\put(1,31){\circle*{1.5}}
\put(5,31){\circle*{1.5}}
\put(9,31){\circle*{1.5}}
\put(13,31){\circle*{1.5}}
\put(17,35){\circle*{1.5}}
\put(21,35){\circle*{1.5}}
\put(17,31){\circle*{1.5}}
\put(21,31){\circle*{1.5}}
\put(25,31){\circle*{1.5}}
\put(29,31){\circle*{1.5}}
\put(33,35){\circle*{1.5}}
\put(37,35){\circle*{1.5}}
\put(33,31){\circle*{1.5}}
\put(37,31){\circle*{1.5}}
\put(41,31){\circle*{1.5}}
\put(45,31){\circle*{1.5}}
\put(49,35){\circle*{1.5}}
\put(53,35){\circle*{1.5}}
\put(49,31){\circle*{1.5}}
\put(53,31){\circle*{1.5}}
\put(57,31){\circle*{1.5}}
\put(61,31){\circle*{1.5}}
\put(65,35){\circle*{1.5}}
\put(69,35){\circle*{1.5}}
\put(65,31){\circle*{1.5}}
\put(69,31){\circle*{1.5}}
\put(73,31){\circle*{1.5}}
\put(77,31){\circle*{1.5}}
\put(81,35){\circle*{1.5}}
\put(85,35){\circle*{1.5}}
\put(81,31){\circle*{1.5}}
\put(85,31){\circle*{1.5}}
\put(89,31){\circle*{1.5}}
\put(93,31){\circle*{1.5}}
\put(7,28){\makebox(0,0)[cc]{\footnotesize$H_{13}$}}
\put(23,28){\makebox(0,0)[cc]{\footnotesize$H_{14}$}}
\put(39,28){\makebox(0,0)[cc]{\footnotesize$H_{15}$}}
\put(55,28){\makebox(0,0)[cc]{\footnotesize$H_{16}$}}
\put(71,28){\makebox(0,0)[cc]{\footnotesize$H_{17}$}}
\put(87,28){\makebox(0,0)[cc]{\footnotesize$H_{18}$}}
\put(1,21){\line(1,0){12}}
\put(17,21){\line(1,0){12}}
\put(33,21){\line(1,0){12}}
\put(49,21){\line(1,0){12}}
\put(65,21){\line(1,0){12}}
\put(81,21){\line(1,0){12}}
\put(1,21){\line(0,1){4}}
\put(1,25){\line(1,-1){4}}
\put(5,21){\line(0,1){4}}
\put(5,25){\line(-1,-1){4}}
\put(5,25){\line(1,-1){4}}
\put(21,21){\line(0,1){4}}
\put(21,25){\line(1,0){4}}
\put(25,25){\line(-1,-1){4}}
\put(25,25){\line(1,-1){4}}
\put(37,21){\line(0,1){4}}
\put(37,25){\line(1,0){4}}
\put(41,25){\line(-1,-1){4}}
\put(41,25){\line(0,-1){4}}
\put(41,25){\line(1,-1){4}}
\put(49,25){\line(0,-1){4}}
\put(49,25){\line(1,-1){4}}
\put(53,21){\line(1,1){4}}
\put(57,25){\line(1,-1){4}}
\put(49,25){\line(1,0){8}}
\put(65,21){\line(0,1){4}}
\put(65,25){\line(1,-1){4}}
\put(69,21){\line(1,1){4}}
\put(73,25){\line(1,-1){4}}
\put(73,25){\line(0,-1){4}}
\put(65,25){\line(1,0){8}}
\put(85,21){\line(0,1){4}}
\put(85,25){\line(1,0){4}}
\put(89,25){\line(0,-1){4}}
\put(1,25){\circle*{1.5}}
\put(5,25){\circle*{1.5}}
\put(1,21){\circle*{1.5}}
\put(5,21){\circle*{1.5}}
\put(9,21){\circle*{1.5}}
\put(13,21){\circle*{1.5}}
\put(21,25){\circle*{1.5}}
\put(25,25){\circle*{1.5}}
\put(17,21){\circle*{1.5}}
\put(21,21){\circle*{1.5}}
\put(25,21){\circle*{1.5}}
\put(29,21){\circle*{1.5}}
\put(37,25){\circle*{1.5}}
\put(41,25){\circle*{1.5}}
\put(33,21){\circle*{1.5}}
\put(37,21){\circle*{1.5}}
\put(41,21){\circle*{1.5}}
\put(45,21){\circle*{1.5}}
\put(49,25){\circle*{1.5}}
\put(57,25){\circle*{1.5}}
\put(49,21){\circle*{1.5}}
\put(53,21){\circle*{1.5}}
\put(57,21){\circle*{1.5}}
\put(61,21){\circle*{1.5}}
\put(65,25){\circle*{1.5}}
\put(73,25){\circle*{1.5}}
\put(65,21){\circle*{1.5}}
\put(69,21){\circle*{1.5}}
\put(73,21){\circle*{1.5}}
\put(77,21){\circle*{1.5}}
\put(85,25){\circle*{1.5}}
\put(89,25){\circle*{1.5}}
\put(81,21){\circle*{1.5}}
\put(85,21){\circle*{1.5}}
\put(89,21){\circle*{1.5}}
\put(93,21){\circle*{1.5}}
\put(7,18){\makebox(0,0)[cc]{\footnotesize$H_{19}$}}
\put(23,18){\makebox(0,0)[cc]{\footnotesize$H_{20}$}}
\put(39,18){\makebox(0,0)[cc]{\footnotesize$H_{21}$}}
\put(55,18){\makebox(0,0)[cc]{\footnotesize$H_{22}$}}
\put(71,18){\makebox(0,0)[cc]{\footnotesize$H_{23}$}}
\put(87,18){\makebox(0,0)[cc]{\footnotesize$H_{24}$}}
\put(1,11){\line(1,0){12}}
\put(17,11){\line(1,0){12}}
\put(33,11){\line(1,0){12}}
\put(49,11){\line(1,0){12}}
\put(65,11){\line(1,0){12}}
\put(81,11){\line(1,0){12}}
\put(5,11){\line(0,1){4}}
\put(5,15){\line(1,0){8}}
\put(13,15){\line(0,-1){4}}
\put(9,11){\line(1,1){4}}
\multiput(33,11)(.03333333,.06666667){60}{\line(0,1){.06666667}}
\put(35,15){\line(3,-2){6}}
\multiput(39,15)(.03333333,-.06666667){60}{\line(0,-1){.06666667}}
\put(33,11){\line(3,2){6}}
\put(35,15){\line(1,0){4}}
\put(17,11){\line(0,1){4}}
\put(17,15){\line(1,-1){4}}
\put(29,15){\line(0,-1){4}}
\put(29,15){\line(-1,-1){4}}
\put(17,15){\line(1,0){12}}
\multiput(49,11)(.03333333,.06666667){60}{\line(0,1){.06666667}}
\multiput(51,15)(.03333333,-.06666667){60}{\line(0,-1){.06666667}}
\put(51,15){\line(3,-2){6}}
\multiput(57,11)(-.03333333,.06666667){60}{\line(0,1){.06666667}}
\multiput(55,15)(-.03333333,-.06666667){60}{\line(0,-1){.06666667}}
\put(55,15){\line(-3,-2){6}}
\multiput(67,15)(-.03333333,-.06666667){60}{\line(0,-1){.06666667}}
\put(67,15){\line(3,-2){6}}
\multiput(73,11)(-.03333333,.06666667){60}{\line(0,1){.06666667}}
\multiput(71,15)(-.03333333,-.06666667){60}{\line(0,-1){.06666667}}
\put(65,11){\line(3,2){6}}
\multiput(81,11)(.03333333,.06666667){60}{\line(0,1){.06666667}}
\put(83,15){\line(3,-2){6}}
\multiput(89,11)(-.03333333,.06666667){60}{\line(0,1){.06666667}}
\put(87,15){\line(-3,-2){6}}
\multiput(87,15)(-.03333333,-.06666667){60}{\line(0,-1){.06666667}}
\put(83,15){\line(1,0){4}}
\put(5,15){\circle*{1.5}}
\put(13,15){\circle*{1.5}}
\put(1,11){\circle*{1.5}}
\put(5,11){\circle*{1.5}}
\put(9,11){\circle*{1.5}}
\put(13,11){\circle*{1.5}}
\put(17,15){\circle*{1.5}}
\put(29,15){\circle*{1.5}}
\put(17,11){\circle*{1.5}}
\put(21,11){\circle*{1.5}}
\put(25,11){\circle*{1.5}}
\put(29,11){\circle*{1.5}}
\put(35,15){\circle*{1.5}}
\put(39,15){\circle*{1.5}}
\put(33,11){\circle*{1.5}}
\put(37,11){\circle*{1.5}}
\put(41,11){\circle*{1.5}}
\put(45,11){\circle*{1.5}}
\put(51,15){\circle*{1.5}}
\put(55,15){\circle*{1.5}}
\put(49,11){\circle*{1.5}}
\put(53,11){\circle*{1.5}}
\put(57,11){\circle*{1.5}}
\put(61,11){\circle*{1.5}}
\put(67,15){\circle*{1.5}}
\put(71,15){\circle*{1.5}}
\put(65,11){\circle*{1.5}}
\put(69,11){\circle*{1.5}}
\put(73,11){\circle*{1.5}}
\put(77,11){\circle*{1.5}}
\put(83,15){\circle*{1.5}}
\put(87,15){\circle*{1.5}}
\put(81,11){\circle*{1.5}}
\put(85,11){\circle*{1.5}}
\put(89,11){\circle*{1.5}}
\put(93,11){\circle*{1.5}}
\put(7,8){\makebox(0,0)[cc]{\footnotesize$H_{25}$}}
\put(23,8){\makebox(0,0)[cc]{\footnotesize$H_{26}$}}
\put(39,8){\makebox(0,0)[cc]{\footnotesize$H_{27}$}}
\put(55,8){\makebox(0,0)[cc]{\footnotesize$H_{28}$}}
\put(71,8){\makebox(0,0)[cc]{\footnotesize$H_{29}$}}
\put(87,8){\makebox(0,0)[cc]{\footnotesize$H_{30}$}}
\put(1,1){\line(1,0){12}}
\put(17,1){\line(1,0){12}}
\put(33,1){\line(1,0){12}}
\put(49,1){\line(1,0){12}}
\put(1,1){\line(1,1){4}}
\put(5,5){\line(1,-1){4}}
\put(5,1){\line(1,1){4}}
\put(9,5){\line(1,-1){4}}
\put(17,1){\line(1,1){4}}
\put(21,5){\line(1,-1){4}}
\put(25,1){\line(0,1){4}}
\put(25,5){\line(-1,-1){4}}
\put(25,5){\line(1,-1){4}}
\put(33,1){\line(1,1){4}}
\put(37,5){\line(0,-1){4}}
\put(37,1){\line(1,1){4}}
\put(41,5){\line(0,-1){4}}
\put(37,5){\line(1,-1){4}}
\put(41,5){\line(1,-1){4}}
\put(49,1){\line(1,1){4}}
\put(53,5){\line(1,-1){4}}
\put(53,1){\line(1,1){4}}
\put(57,5){\line(1,-1){4}}
\put(53,5){\line(1,0){4}}
\put(5,5){\circle*{1.5}}
\put(9,5){\circle*{1.5}}
\put(1,1){\circle*{1.5}}
\put(5,1){\circle*{1.5}}
\put(9,1){\circle*{1.5}}
\put(13,1){\circle*{1.5}}
\put(21,5){\circle*{1.5}}
\put(25,5){\circle*{1.5}}
\put(17,1){\circle*{1.5}}
\put(21,1){\circle*{1.5}}
\put(25,1){\circle*{1.5}}
\put(29,1){\circle*{1.5}}
\put(37,5){\circle*{1.5}}
\put(41,5){\circle*{1.5}}
\put(33,1){\circle*{1.5}}
\put(37,1){\circle*{1.5}}
\put(41,1){\circle*{1.5}}
\put(45,1){\circle*{1.5}}
\put(53,5){\circle*{1.5}}
\put(57,5){\circle*{1.5}}
\put(49,1){\circle*{1.5}}
\put(53,1){\circle*{1.5}}
\put(57,1){\circle*{1.5}}
\put(61,1){\circle*{1.5}}
\put(7,-2){\makebox(0,0)[cc]{\footnotesize$H_{31}$}}
\put(23,-2){\makebox(0,0)[cc]{\footnotesize$H_{32}$}}
\put(39,-2){\makebox(0,0)[cc]{\footnotesize$H_{33}$}}
\put(55,-2){\makebox(0,0)[cc]{\footnotesize$H_{34}$}}
\end{picture}
\caption{\small The graphs $H_1$--$H_{34}$.}
\label{Figure-3}
\end{center}
\end{figure}

Without loss of generality, we assume that $G_1$ contains induced $H=P_3=v_1v_2v_3$. Then $G_2$  contains no induced $2K_1$ because $F_6=P_3\vee(2K_1)$ cannot be the induced subgraph of $G$ by Lemma \ref{Lemma-3-1}. This implies that $G_2$ is a complete graph $K_a$ ($a\geq 1$).

Now consider the structure of $G_1$. For any vertex $v\in V(G_1)\setminus V(H)$, we claim that $v$ is adjacent to at least one vertex of $H$, since otherwise $F_7$ will be an induced subgraph of $G$, which is impossible by  Lemma \ref{Lemma-3-1}. Thus, for any $v\in V(G_1)\setminus V(H)$, we have $G_1[\{v_1,v_2,v_3,v\}]\in\{H_1,H_2,H_3,H_4,H_5\}$. Obviously, $G_1[\{v_1,v_2,v_3,v\}]\neq H_1$ because $G$ contains no induced $P_4$. Furthermore, we see that $G_1[\{v_1,v_2,v_3,v\}]\neq H_4$ because $G$ cannot contain $F_6$ as its induced subgraph. Thus $G_1[\{v_1,v_2,v_3,v\}]\in\{H_2,H_3,H_5\}$ and we have the following claim.

\vspace{0.2cm}

\noindent\textbf{Claim 1.1.} For any $v\in V(G_1)\setminus V(H)$, $N_{G_1}(v)\cap V(H)=\{v_2\}$, $\{v_1,v_2\}$, $\{v_2,v_3\}$ or $\{v_1,v_2,v_3\}$.

Denote by $V_1$, $V_2$, $V_3$ and $V_4$ the sets of $v\in V(G_1)\setminus V(H)$ such that  $N_{G_1}(v)\cap V(H)=\{v_2\}$, $\{v_1,v_2\}$, $\{v_1,v_2,v_3\}$ and $\{v_2,v_3\}$, respectively. Then $V(G_1)\setminus V(H)=V_1\cup V_2\cup V_3\cup V_4$. Now we begin to analyse the structure of $G_1$.

If $G_1[V_{1}]$ contains an induced $P_3=u_1u_2u_3$, then $G_1[\{v_1,v_2,u_1,u_2,u_3\}]=F_{7}$, which is impossible by Lemma \ref{Lemma-3-1}.  This implies that $G_1[V_{1}]$ is the disjoint union of some complete graphs. Moreover, we see that $G_1[V_{1}]$ contains no induced $2K_2$ because $F_{4}$ is not an induced subgraph of $G$. Therefore, if $V_1\neq\emptyset$ then  $G_1[V_{1}]$ is one of the following graphs: $K_a$ ($a\geq 2$), $K_a^{c}$ ($a\geq 1$) and $K_a\cup K_b^c$ ($a\geq 2$, $b\geq 1$).  

For any $u,v\in V_2$, if $u$ and $v$ are not adjacent, then $G_{1}[v_1,v_2,v_3,u,v]=F_{7}$, a contradiction. Thus $G_{1}[V_2]$ (if $V_2\neq \emptyset$) is a complete graph, and so is $G_1[V_4]$ by the symmetry. Similarly, we see that $G_{1}[V_3]$ (if $V_3\neq \emptyset$) is also a complete graph because  $F_6$ cannot be the induced subgraph of $G$.

For any $u\in V_1$ and $v\in V_2$ (resp. $v\in V_4$), if $u$ and $v$ are adjacent, then $G_1[\{v_1,v_2,v_3,u,v\}]=F_{7}$, a contradiction. Thus there are no edges connecting $V_{1}$ and $V_{2}\cup V_{4}$. Moreover, every vertex of $V_{1}$ is adjacent to every vertex of $V_{3}$ again because $F_{7}$ cannot be the induced subgraph of $G$.

For any $u\in V_2$ (resp. $u\in V_4$) and $v\in V_3$, if $u,v$ are not adjacent, then $G_1[\{v_1,v_2,v_3,u,v\}]=F_{3}$, which is a contradiction. Thus  every vertex of $V_{2}\cup V_4$ is adjacent to every vertex of $V_{3}$. Moreover, we claim that there are no edges connecting $V_2$ and $V_4$ again because $G$ contains no induced $F_3$. 
 
By the definition of $V_i$ ($1\leq i\leq 4$), we see that $v_1$ is adjacent to every vertex of $V_{2}\cup V_{3}$ but none of $V_{1}\cup V_4$,  $v_2$ is adjacent to every vertex of $V_{1}\cup V_2\cup V_3\cup V_4$, and $v_3$ is adjacent to every vertex of $V_{3}\cup V_{4}$ but none of $V_{1}\cup V_2$. Put $V_2'=V_2\cup \{v_1\}$, $V_3'=V_3\cup \{v_2\}$ and $V_4'=V_4\cup \{v_3\}$. Then $V(G_1)=V_1\cup V_2'\cup V_3'\cup V_4'$. 

Summarizing  above results, we  see that  $G_{1}[V_1]$ (if $V_1\neq\emptyset$) is of the from $K_a$ ($a\geq 2$), $K_a^{c}$ ($a\geq 1$) or $K_a\cup K_b^c$ ($a\geq 2$, $b\geq 1$), and $G_{1}[V_i']$ ($|V_i'|\geq 1$) is a complete graph for $i=2,3,4$; every vertex of $V_3'$ is adjacent to every vertex of $V_1\cup V_2'\cup V_4'$ and there are no edges connecting $V_1$, $V_2'$ and $V_4'$. Therefore, $G_1=G_1[V_3']\vee G_1[V_1\cup V_2'\cup V_4']$ is of one form listed below: $K_a\vee(K_b\cup K_c)$ with $a,b,c\geq 1$, $K_a\vee (K_b\cup K_c\cup K_d)$ with $a,c,d\geq 1$ and $b\geq 2$, $K_a\vee (K_b^c\cup K_c\cup K_d)$ with $a,b,c,d\geq 1$ or $K_a\vee (K_b\cup K_c^c\cup K_d\cup K_e)$ with $a,c,d,e\geq 1$ and $b\geq 2$. Considering that $G$ (and so $G_1$) cannot contain $F_4$ as its induced subgraph, we have $G_1=K_a\vee (K_b \cup K_c)$ ($a,b,c\geq 1$),  $K_a\vee (K_b \cup K_c^c)$ ($a\geq 1$, $b,c\geq 2$) or $K_a\vee K_b^c$ ($a\geq 1$, $b\geq 3$). Recalling that $G_2$ is a complete graph and $G=G_1\vee G_2$, we obtain $G=K_a\vee (K_b \cup K_c)$ ($a\geq 2$, $b,c\geq 1$),  $K_a\vee (K_b \cup K_c^c)$ ($a\geq 2$, $b,c\geq 2$) or $K_a\vee K_b^c$ ($a\geq 2$, $b\geq 3$). Thus we also have $G\in\{I_i\mid 1\leq i\leq 7\}$.

\vspace{0.2cm}

\noindent\textbf{Case 2.} $d(G)=3$.

Let $H=P_4=v_1v_2v_3v_4$ be a diameter path of $G$. Then $H$ is an induced subgraph of $G$ and $D(H)=D_G(\{v_1,v_2,v_3,v_4\})$ is a principal submatrix of $D(G)$. Firstly, we have the following claim.

\vspace{0.2cm}

\noindent\textbf{Claim 2.1.}  $d(v,H)=1$ for any  $v\in V(G)\setminus V(H)$.

If not, we have  $2\leq d(v,H)\leq 3$ since $d(G)=3$. Let $d_i=d(v,v_i)$ for $i=1,2,3,4$. Then $d_i\in\{2,3\}$ for each $i$, and the  principal submatrix of $G$ corresponding to $G[\{v_1,v_2,v_3,v_4,v\}]$ is of the form
$$
D_G(\{v_1,v_2,v_3,v_4,v\})=\left[\begin{matrix}
0&1&2&3&d_1\\
1&0&1&2&d_2\\
2&1&0&1&d_3\\
3&2&1&0&d_4\\
d_1&d_2&d_3&d_4&0
\end{matrix}\right].
$$
In Tab. \ref{Table-2}, we list the possible values of the second least eigenvalue of $D_G(\{v_1,v_2,v_3,$ $v_4,v\})$, which are all less than $-2$. Then from Lemma \ref{Lemma-2-4} we get $\partial_{n-1}(G)\leq \partial_4(D_G(\{v_1,v_2,v_3,v_4,v\})<-2$, contrary to $\partial_{n-1}(G)\geq-2$. Hence, each vertex in $V(G)\setminus V(H)$ must be adjacent to at least one vertex of $H$.

\begin{table}[t]
\small
\caption{The second least eigenvalue of $D_G(\{v_1,v_2,$ $v_3,v_4,v\})$.}
\vspace{-0.5cm}
\begin{center}
\scriptsize
\begin{tabular}{cccccccc}
\hline
$(d_1,d_2,d_3,d_4)$&$\partial_4$&$(d_1,d_2,d_3,d_4)$&$\partial_4$&$(d_1,d_2,d_3,d_4)$&$\partial_4$&$(d_1,d_2,d_3,d_4)$&$\partial_4$\\
\hline
$(2,2,2,2)$&$-2.3956$&$(2,2,2,3)$&$-2.3810$&$(2,2,3,2)$&$-3.0586$&$(2,2,3,3)$&$-2.6028$\\
$(2,3,2,2)$&$-3.0586$&$(2,3,2,3)$&$-3.1163$&$(2,3,3,2)$&$-3.4142$&$(2,3,3,3)$&$-3.1014$\\
$(3,2,2,2)$&$-2.3810$&$(3,2,2,3)$&$-3.1436$&$(3,2,3,2)$&$-3.1163$&$(3,2,3,3)$&$-3.2798$\\
$(3,3,2,2)$&$-2.6028$&$(3,3,2,3)$&$-3.2798$&$(3,3,3,2)$&$-3.1014$&$(3,3,3,3)$&$-3.4142$\\
\hline
\end{tabular}
\end{center}
\label{Table-2}
\end{table}

Note that $d(v_1,v_4)=3$.  From Claim 2.1 and  the symmetry  of $v_1$ and $v_4$ (resp. $v_2$ and $v_3$), for any $v\in V(G)\setminus V(H)$,  we can suppose $G[\{v_1,v_2,v_3,v_4,v\}]\in\{H_6,H_7,H_8,H_9,H_{10},H_{11}\}$ (see Fig. \ref{Figure-3}). If $G[\{v_1,v_2,v_3,v_4,v\}]=H_6$, then $d(v,v_1)=1$, $d(v,v_2)=2$, and $d(v,v_3),d(v,v_4)\in\{2,3\}$. Thus $D(G)$ has $D_G(\{v_1,v_2,v_3,v_4,v\})\in\{A_1,A_2,A_3,A_4\}$ as its  principal submatrix, which is impossible by Lemma  \ref{Lemma-3-2}.  Similarly, if $G[\{v_1,v_2,v_3,v_4,v\}]=H_{9}$, then the corresponding  principal submatrix is given by $D_G(\{v_1,v_2,v_3,v_4,v\})=A_5$, a contradiction. Hence, $G[\{v_1,v_2,v_3,v_4,v\}]\in\{H_7,H_8,H_{10},H_{11}\}$ for any $v\in V(G)\setminus V(H)$. Again by considering the symmetry  of $v_1$ and $v_4$ (resp. $v_2$ and $v_3$), we have the following claim.

\vspace{0.2cm}

\noindent\textbf{Claim 2.2.}  For any $v\in V(G)\setminus V(H)$, $N_G(v)\cap V(H)=\{v_2\}$, $\{v_3\}$, $\{v_1,v_2\}$, $\{v_3,v_4\}$, $\{v_1,v_3\}$, $\{v_2,v_4\}$, $\{v_1,v_2,v_3\}$ or $\{v_2,v_3,v_4\}$.

Denote by $V_{11}$, $V_{12}$, $V_{21}$, $V_{22}$, $V_{31}$, $V_{32}$, $V_{41}$ and $V_{42}$ the sets of $v\in V(G)\setminus V(H)$ such that  $N_G(v)\cap V(H)=\{v_2\}$, $\{v_1,v_2\}$, $\{v_1,v_3\}$, $\{v_1,v_2,v_3\}$, $\{v_2,v_4\}$, $\{v_2,v_3,v_4\}$, $\{v_3\}$ and $\{v_3,v_4\}$, respectively. Let $V_i=V_{i1}\cup V_{i2}$ for $i=1,2,3,4$. Then  $V(G)\setminus V(H)=V_1\cup V_2\cup V_3\cup V_4$.

For any $u,v\in V_{11}$, if $u$  and $v$ are adjacent, then $G[\{v_1,v_2,v_3,v_4,u,v\}]=H_{12}$ (see Fig. \ref{Figure-3}), and the corresponding principal submatrix $D_G(\{v_1,v_2,v_3,v_4,u,v\})$ belongs to $\{A_6,A_{7},A_{8},A_{9}\}$ because $d(u,v_1)=d(v,v_1)=d(u,v_3)=d(v,v_3)=2$, $d(u,v_2)=d(v,v_2)=1$, $d(u,v_4),d(v,v_4)\in \{2,3\}$ and $d(u,v)=1$, which contradicts Lemma \ref{Lemma-2-8}. Thus  $V_{11}$ is an independent set, and so is $V_{41}$ by the symmetry. Similarly, if $u,v\in V_{12}$ are not adjacent, then $G[\{v_1,v_2,v_3,v_4,u,v\}]=H_{13}$ and the corresponding principal submatrix $D_G(\{v_1,v_2,v_3,v_4,u,v\})$ belongs to $\{A_{10},A_{11},A_{12},A_{13}\}$, implying that $V_{12}$ is a clique and so is $V_{42}$.  Furthermore, if neither $V_{11}$ nor $V_{12}$ is empty, then $H_{14}$ or $H_{15}$ is an induced subgraph of $G$, and the  corresponding principal submatrix is one of  $\{A_i\mid 14\leq i\leq 21\}$, a contradiction. Thus  at least one of  $V_{11}$ and $V_{12}$ (resp.  $V_{41}$ and $V_{42}$ by the symmetry) is empty.

For any $u\in V_{11}$ and $v\in V_{21}$, if $u$ and $v$ are not adjacent, then $G[\{v_1,v_2,v_3,v_4,u,$ $v\}]=H_{16}$ and $D_G(\{v_1,v_2,v_3,v_4,u,v\})\in\{A_{22},A_{23},A_{24},A_{25}\}$, which is impossible and so each vertex of $V_{11}$ is adjacent to every vertex of $V_{12}$.  Similarly, if $u\in V_{11}$ and  $v\in V_{22}$ are not adjacent, then $G[\{v_1,v_2,v_3,v_4,u,v\}]=H_{17}$ and $D_G(\{v_1,v_2,v_3,v_4,u,v\})=\{A_{26},A_{27}\}$; if $u\in V_{12}$ and $v\in V_{21}$ are not adjacent, then $G[\{v_1,v_2,v_3,v_4,u,v\}]=H_{18}$ and $D_G(\{v_1,v_2,v_3,v_4,u,v\})\in\{A_{28},A_{29}\}$; if $u\in V_{12}$ and $v\in V_{22}$ are not adjacent, then $G[\{v_1,v_2,v_3,v_4,u,v\}]=H_{19}$ and $D_G(\{v_1,v_2,v_3,v_4,$ $u,v\})\in\{A_{30},A_{31}\}$. Thus all these cases are impossible, and we conclude that every vertex of $V_{1}=V_{11}\cup V_{12}$  is adjacent to every vertex of $V_2=V_{21}\cup V_{22}$, and by the symmetry, every vertex of $V_{4}=V_{41}\cup V_{42}$  is adjacent to every vertex of $V_3=V_{31}\cup V_{32}$. 

As above, if $u\in V_1=V_{11}\cup V_{12}$ and $v\in V_3=V_{31}\cup V_{32}$  are adjacent, then $G[\{v_1,v_2,v_3,v_4,u,v\}]\in\{H_{20},H_{21},H_{22},H_{23}\}$ and $D_G(\{v_1,v_2,v_3,v_4,u,v\})\in\{A_{32},$ $A_{33},A_{34},A_{35}\}$;  if $u\in V_1=V_{11}\cup V_{12}$ and $v\in V_4=V_{41}\cup V_{42}$  are adjacent, then $G[\{v_1,v_2,v_3,v_4,u,v\}]\in\{H_{24},H_{25},H_{26}\}$ and $D_G(\{v_1,v_2,v_3,v_4,u,v\})\in\{A_i\mid 36\leq i\leq 41\}$. Therefore, there are no edges in $G$ connecting $V_1$ and $V_3\cup V_4$, and symmetrically, there  are no edges connecting $V_4$ and $V_2\cup V_3$.

For any $u,v\in V_{21}$, if $u$ and $v$ are adjacent, then $G[\{v_1,v_2,v_3,v_4,u,v\}]=H_{27}$, and the corresponding principal submatrix is $D_G(\{v_1,v_2,v_3,v_4,u,v\})=A_{43}$, a contradiction. Thus  $V_{21}$ is an independent set and so is $V_{31}$ by the symmetry. Similarly, if $u,v\in V_{22}$ are not adjacent, then $G[\{v_1,v_2,v_3,v_4,u,v\}]=H_{28}$ and the corresponding principal submatrix $D_G(\{v_1,v_2,v_3,v_4,u,v\})$ is equal to $A_{44}$, which implies that $V_{22}$ is a clique and so is $V_{32}$. Furthermore, if neither $V_{21}$ nor $V_{22}$ is empty, then $H_{29}$ or $H_{30}$ is an induced subgraph of $G$, and the  corresponding principal submatrix is $A_{45}$ or $A_{46}$. Thus  at least one of  $V_{21}$ and $V_{22}$ (resp.  $V_{31}$ and $V_{32}$ by the symmetry) is empty.

Also, if $u\in V_2=V_{21}\cup V_{22}$ and $v\in V_3=V_{31}\cup V_{32}$  are not adjacent, then $G[\{v_1,v_2,v_3,v_4,u,v\}]\in\{H_{31},H_{32},H_{33}\}$ and $D_G(\{v_1,v_2,v_3,v_4,u,v\})\in\{A_{47},A_{48},$ $A_{49},A_{50}\}$, which implies that every vertex of $V_2$ is adjacent to every vertex of $V_3$. Moreover, we claim that either $V_{21}$ or $V_{31}$ is empty, since otherwise $H_{34}$ will be an induced subgraph of $G$ and the corresponding principal submatrix is $A_{51}$, a contradiction.

Summarizing above results, we have the following claim.

\vspace{0.2cm}

\noindent\textbf{Claim 2.3.}  The graph $G$ has the  properties P1--P4:
\begin{enumerate}[(P1)]
\item  every vertex of $V_2$ is adjacent to every vertex of $V_1$ and $V_3$, and  every vertex of $V_3$ is adjacent to every vertex of $V_2$ and $V_4$; 
\vspace{-0.2cm}
\item  there are no edges connecting $V_1$ and $V_3\cup V_4$, and   $V_2$ and $V_4$; 
\vspace{-0.2cm}
\item  for each $1\leq i\leq 4$, $V_{i1}=\emptyset$ or $V_{i2}=\emptyset$, and if $V_{i1}\neq \emptyset$ (resp. $V_{i2}\neq \emptyset$) then $V_{i1}$ (resp. $V_{i2}$) is an independent set (resp. clique);
\vspace{-0.2cm}
\item  $V_{21}=\emptyset$ or $V_{31}=\emptyset$.
\end{enumerate}
 
Not put $V_i'=V_i\cup\{v_i\}$ for $i=1,2,3,4$. Then $V(G)=V_1'\cup V_2'\cup V_3' \cup V_4'$.  From the definition of $V_i=V_{i1}\cup V_{i2}$ and $V_i'$ ($1\leq i\leq 4$), we see that $v_i$ is adjacent to every vertex of $V_{i2}$ but none of $V_{i1}$, $v_1$ (resp. $v_4$) is adjacent to  every vertex of $V_2'$ (resp. $V_3'$), $v_2$  (resp.  $v_3$) is adjacent to  every vertex of $V_1'\cup V_3'$ (resp. $V_2'\cup V_4'$). Combining this with Claim 1.3, we may conclude that $G=P_4[G_1,G_2,G_3,G_4]$, where $G_i=G[V_i']$ is a complete graph or a union of some isolated vertices for  $1\leq i\leq 4$, and $G_2, G_3$ cannot be the union of some isolated vertices at the same time if $|V_2'|,|V_3'|\geq 2$.  By the symmetry of $V_1'$ and $V_4'$ (resp. $V_2'$ and $V_3'$), without loss of generality, we can suppose that $G$ is one of the following graphs:
$J_1=P_4[K_{a}^c,K_{b},K_{c}^c,K_{d}^c]$, $J_2=P_4[K_{a}^c,K_{b},K_{c},K_{d}^c]$, $J_3=P_4[K_{a}^c,K_{b},K_{c}^c,K_{d}]$, $J_4=P_4[K_{a}^c,K_{b}^c,K_{c},K_{d}]$, $J_5=P_4[K_{a}^c,K_{b},K_{c},K_{d}]$, $J_6=P_4[K_{a},K_{b},K_{c}^c,K_{d}]$ and $J_7=P_4[K_{a},K_{b},K_{c},K_{d}]$, where $a,b,c,d\geq 1$. 

We complete the proof.
\end{proof}
\renewcommand{\arraystretch}{1.0}
\begin{table}[t]
\caption{The $D$-polynomials of $I_1$--$I_7$ and $J_1$--$J_7$.}
\centering
\scriptsize
\begin{tabular}{ll}
 \hline
$G$ &$\Phi_G(x)$\\
 \hline
$I_1=K_n$ ($n\geq 4$) &$(x\!-\!n\!+\!1)(x\!+\!1)^{n\!-\!1}$;\\
$I_2=K_a\vee K_b^c$ ($a\geq 2,b\geq 1$) &$(x\!+\!1)^{a\!-\!1}(x\!+\!2)^{b\!-\!1}[x^2 \!+\! (3 \!-\! 2b \!-\! a)x \!-\! 2a \!-\! 2b \!+\! ab \!+\! 2]$;\\
$I_3=K_a\vee (K_b\cup K_c)$ ($a,b,c\geq 2$)&\tabincell{l}{$(x\!+\!1)^{a\!+\!b\!+\!c\!-\!3}[x^3 \!+\! (3 \!-\! b \!-\! c \!-\! a)x^2 \!+\! (3 \!-\! 2b \!-\! 2c \!-\! 3bc \!-\! 2a)x \!-\! a \!-\! b $\\ $-\!c\!-\! 3bc \!+\! abc \!+\! 1]$;}\\
$I_4=K_a\vee (K_b\cup K_c^c)$ ($a,b\geq 2$, $c\geq 1$)&\tabincell{l}{$(x\!+\!1)^{a\!+\!b\!-\!2}(x\!+\!2)^{c\!-\!1}[x^3 \!+\! (4 \!-\! b \!-\! 2c \!-\! a)x^2 \!+\! (ac \!-\! 3b \!-\! 4c \!-\! 3a \!-\! 2bc$\\ $+\! 5)x\!-\! 2a \!-\! 2b \!-\! 2c \!+\! ac\!-\! 2bc \!+\! abc \!+\! 2]$;}\\
$I_5=K_a^c\vee K_{b}^c$ ($a,b\geq 1$)&$(x\!+\!2)^{a\!+\!b\!-\!2}[x^2 \!+\! (4 \!-\! 2b \!-\! 2a)x \!-\! 4a \!-\! 4b \!+\! 3ab \!+\! 4]$;\\
$I_6=K_a^c\vee (K_b\cup K_{c})$  ($a\geq 1$, $b,c\geq 2$)&\tabincell{l}{$(x\!+\!1)^{b\!+\!c\!-\!2}(x\!+\!2)^{a\!-\!1}[x^3 \!+\! (4 \!-\! b \!-\! c \!-\! 2a)x^2 \!+\! (ab \!-\! 3b \!-\! 3c \!-\! 4a \!+\! ac$\\ $-\! 3bc\!+\! 5)x \!-\! 2a \!-\! 2b \!-\! 2c\!+\! ab \!+\! ac \!-\! 6bc \!+\! 4abc \!+\! 2]$;}\\
$I_7=K_a^c\vee (K_b\cup K_{c}^c)$ ($a,c\geq 1$, $b\geq 2$)&\tabincell{l}{$(x\!+\!1)^{b\!-\!1}(x\!+\!2)^{a\!+\!c\!-\!2}[x^3 \!+\! (5 \!-\! b \!-\! 2c \!-\! 2a)x^2 \!+\! (ab \!-\! 4b \!-\! 6c \!-\! 6a \!+\! 3ac$\\ $-\! 2bc \!+\! 8)x \!-\! 4a \!-\! 4b \!-\! 4c\!+\! 2ab \!+\! 3ac \!-\! 4bc \!+\! 3abc \!+\! 4]$;}\\
$J_1=P_4[K_{a}^c,K_{b},K_{c}^c,K_{d}^c]$&\tabincell{l}{$(x\!+\!1)^{b\!-\!1}(x\!+\!2)^{a\!+\!c\!+\!d\!-\!3}[x^4 \!+\! (7 \!-\! b \!-\! 2c \!-\! 2d \!-\! 2a)x^3 \!+\! (ab \!-\! 6b \!-\! 10c$\\ $-\! 10d\!-\! 10a \!-\! 5ad \!+\! bc \!-\! 2bd\!+\! 3cd \!+\! 18)x^2\!+\! (4ab \!-\! 12b \!-\! 16c \!-\! 16d \!-\! 16a$\\ $-\! 15ad\!+\! 4bc \!-\! 8bd \!+\! 9cd \!+\! 3abd \!+\! 8acd \!+\! 3bcd \!+\! 20)x\!-\! 8a \!-\! 8b \!-\! 8c \!-\! 8d$\\ $+\! 4ab\!-\! 10ad\!+\! 4bc \!-\! 8bd \!+\! 6cd \!+\! 6abd \!+\! 8acd \!+\! 6bcd \!-\! 4abcd \!+\! 8]$;}\\
$J_2=P_4[K_{a}^c,K_{b},K_{c},K_{d}^c]$&\tabincell{l}{$(x\!+\!1)^{b\!+\!c\!-\!2}(x\!+\! 2)^{a\!+\!d\!-\!2}[x^4 \!+\! (6 \!-\! b \!-\! c \!-\! 2d \!-\! 2a)x^3 \!+\! (ab \!-\! 5b \!-\! 5c \!-\! 8d$\\ $-\! 8a \!-\! 2ac \!-\! 5ad \!-\! 2bd\!+\! cd \!+\! 13)x^2\!+\! (3ab \!-\! 8b \!-\! 8c \!-\! 10d \!-\! 10a \!-\! 6ac$\\ $ \!-\! 10ad\!-\! 6bd \!+\! 3cd \!+\! abc \!+\! 3abd \!+\! 3acd \!+\! bcd \!+\! 12)x\!-\! 4a \!-\! 4b \!-\! 4c \!-\! 4d$\\ $+\! 2ab\!-\! 4ac \!-\! 5ad \!-\! 4bd \!+\! 2cd \!+\! 2abc \!+\! 3abd \!+\! 3acd \!+\! 2bcd \!-\! abcd \!+\! 4]$;}\\
$J_3=P_4[K_{a}^c,K_{b},K_{c}^c,K_{d}]$&\tabincell{l}{$(x\!+\!1)^{b\!+\!d\!-\!2}(x\!+\!2)^{a\!+\!c\!-\!2}[x^4 \!+\! (6 \!-\! b \!-\! 2c \!-\! d \!-\! 2a)x^3 \!+\! (ab \!-\! 5b \!-\! 8c \!-\! 5d$\\ $-\! 8a \!-\! 7ad \!+\! bc \!-\! 3bd\!+\! cd \!+\! 13)x^2\!+\! (3ab \!-\! 8b \!-\! 10c \!-\! 8d \!-\! 10a \!-\! 21ad$\\ $+\! 3bc \!-\! 12bd \!+\! 3cd \!+\! 4abd \!+\! 8acd \!+\! 4bcd \!+\! 12)x\!-\! 4a \!-\! 4b \!-\! 4c \!-\! 4d$\\ $+\! 2ab \!-\! 14ad \!+\! 2bc \!-\! 12bd \!+\! 2cd \!+\! 8abd \!+\! 8acd \!+\! 8bcd \!-\! 4abcd \!+\! 4]$;}\\
$J_4=P_4[K_{a}^c,K_{b}^c,K_{c},K_{d}]$&\tabincell{l}{$(x\!+\!1)^{c\!+\!d\!-\!2}(x\!+\!2)^{a\!+\!b\!-\!2}[x^4 \!+\! (6 \!-\! 2b \!-\! c \!-\! d \!-\! 2a)x^3 \!+\! (3ab \!-\! 8b \!-\! 5c \!-\! 5d$\\ $-\! 8a \!-\! 2ac \!-\! 7ad \!+\! bc\!-\! 2bd \!+\! 13)x^2\!+\! (6ab \!-\! 10b \!-\! 8c \!-\! 8d \!-\! 10a \!-\! 6ac$\\ $-\! 21ad \!+\! 3bc \!-\! 6bd \!+\! 3abc \!+\! 11abd \!+\! acd\!+\! bcd \!+\! 12)x\!-\! 4a \!-\! 4b \!-\! 4c \!-\! 4d$\\ $+\! 3ab \!-\! 4ac \!-\! 14ad \!+\! 2bc \!-\! 4bd \!+\! 3abc \!+\! 11abd \!+\! 2acd \!+\! 2bcd \!-\! abcd \!+\! 4]$;}\\
$J_5=P_4[K_{a}^c,K_{b},K_{c},K_{d}]$&\tabincell{l}{$(x\!+\!1)^{b\!+\!c\!+\!d\!-\!3}(x\!+\!2)^{a\!-\!1}[x^4 \!+\! (5 \!-\! b \!-\! c \!-\! d \!-\! 2a)x^3 \!+\! (ab \!-\! 4b \!-\! 4c \!-\! 4d$\\ $-\! 6a \!-\! 2ac \!-\! 7ad\!-\! 3bd \!+\! 9)x^2\!+\! (2ab \!-\! 5b \!-\! 5c \!-\! 5d \!-\! 6a \!-\! 4ac\!-\! 14ad$\\ $-\! 9bd \!+\! abc \!+\! 4abd \!+\! acd \!+\! bcd \!+\! 7)x\!-\! 2a \!-\! 2b \!-\! 2c \!-\! 2d\!+\! ab \!-\! 2ac$\\ $-\! 7ad \!-\! 6bd \!+\! abc \!+\! 4abd \!+\! acd \!+\! 2bcd \!+\! 2]$;}\\
$J_6=P_4[K_{a},K_{b},K_{c}^c,K_{d}]$&\tabincell{l}{$(x\!+\!1)^{a\!+\!b\!+\!d\!-\!3}(x\!+\!2)^{c\!-\!1}[x^4 \!+\! (5 \!-\! b \!-\! 2c \!-\! d \!-\! a)x^3 \!+\! (bc \!-\! 4b \!-\! 6c \!-\! 4d$\\ $-\! 2ac \!-\! 8ad \!-\! 4a \!-\! 3bd\!+\! cd \!+\! 9)x^2\!+\! (2bc \!-\! 5b \!-\! 6c \!-\! 5d \!-\! 4ac \!-\! 24ad$\\$ -\! 5a \!-\! 9bd \!+\! 2cd \!+\! abc \!+\! abd \!+\! 9acd \!+\! 4bcd \!+\! 7)x\!-\! 2a \!-\! 2b \!-\! 2c \!-\! 2d$\\ $-\! 2ac \!-\! 16ad \!+\! bc \!-\! 6bd \!+\! cd \!+\! abc \!+\! 2abd \!+\! 9acd \!+\! 4bcd \!+\! 2]$;}\\
$J_7=P_4[K_{a},K_{b},K_{c},K_{d}]$&\tabincell{l}{$(x\!+\!1)^{a\!+\!b\!+\!c\!+\!d\!-\!4}[x^4 \!+\! (4 \!-\! b \!-\! c \!-\! d \!-\! a)x^3 \!+\! (6 \!-\! 3b \!-\! 3c \!-\! 3d\!-\! 3ac$\\ $-\! 8ad\!-\! 3bd \!-\! 3a)x^2\!+\! (abc \!-\! 3b \!-\! 3c \!-\! 3d \!-\! 6ac \!-\! 16ad \!-\! 6bd \!-\! 3a$\\ $+\! abd \!+\! acd \!+\! bcd \!+\! 4)x\!-\! a \!-\! b \!-\! c \!-\! d\!-\! 3ac \!-\! 8ad \!-\! 3bd \!+\! abc$\\ $+\! abd \!+\! acd \!+\! bcd \!+\! abcd \!+\! 1]$.}\\
\hline
\end{tabular}
\label{Table-3}
\end{table}

\begin{pro}\label{Pro-3-2}
The $D$-polynomials of $I_1$--$I_7$ and $J_1$--$J_7$ (see Proposition \ref{Pro-3-1}) are listed in Tab. \ref{Table-3}.
\end{pro}
\begin{proof}
We only show how to  obtain the  $D$-polynomial of $J_1$. For the remaining graphs, the methods are similar and so we omit the process of computation. 

It is easily seen that $J_1=P_4[K_{a}^c,K_{b},K_{c}^c,K_{d}^c]$ has the distance divisor matrix
$$B_{\Pi}=
\left[\begin{matrix}
2(a-1)&b&2c&3d\\
a&b-1&c&2d\\
2a&b&2(c-1)&d\\
3a&2b&c&2(d-1)\\
\end{matrix} 
\right].
$$
By Lemma \ref{Lemma-2-8}, we have $\Psi_{J_1}(x)=\det(xI-B_\Pi)| \Phi_{J_1}(x)=\det(xI-D(J_1))$, where $\Psi_{J_1}(x)=x^4 \!+\! (7 \!-\! b \!-\! 2c \!-\! 2d \!-\! 2a)x^3 \!+\! (ab \!-\! 6b \!-\! 10c\!-\! 10d\!-\! 10a \!-\! 5ad \!+\! bc \!-\! 2bd\!+\! 3cd \!+\! 18)x^2\!+\! (4ab \!-\! 12b \!-\! 16c \!-\! 16d \!-\! 16a\!-\! 15ad\!+\! 4bc \!-\! 8bd \!+\! 9cd \!+\! 3abd \!+\! 8acd \!+\! 3bcd \!+\! 20)x\!-\! 8a \!-\! 8b \!-\! 8c \!-\! 8d\!+\! 4ab\!-\! 10ad\!+\! 4bc \!-\! 8bd \!+\! 6cd \!+\! 6abd \!+\! 8acd \!+\! 6bcd \!-\! 4abcd \!+\! 8$. Furthermore, from Lemma \ref{Lemma-2-7} we know that $-1$ and $-2$ are $D$-eigenvalues of $J_1$ with multiplicities at least $b-1$ and $a+c+d-3$, respectively. Thus the  $D$-polynomial of $J_1$ is equal to $\Phi_{J_1}(x)=(x+1)^{b-1}(x+2)^{a+c+d-3}\Psi_{J_1}(x)$ since the constructed eigenvectors we use to prove Lemma \ref{Lemma-2-7} are of the second kind according to  Lemma \ref{Lemma-2-8}.
\end{proof}

Combining Propositions \ref{Pro-3-1} and \ref{Pro-3-2}, we now give the main result of this section.
\begin{thm}\label{Thm-3-1}
Let $G$ be a connected graph on $n$ vertices. Then $\partial_3(G)\leq -1$ and $\partial_{n-1}(G)\geq -2$ if and only if 
\begin{enumerate}[(i)]
\vspace{-0.2cm}
\item $d(G)\leq 2$ and $G\in\{I_i\mid 1\leq i\leq 7\}$, where $I_1=K_n$ ($n\geq 4$), $I_2=K_a\vee K_b^c$ ($a\geq 2,b\geq 1$), $I_3=K_a\vee (K_b\cup K_c)$ ($a,b,c\geq 2$), $I_4=K_a\vee (K_b\cup K_c^c)$ ($a,b\geq 2$, $c\geq 1$), $I_5=K_a^c\vee K_{b}^c$ ($a,b\geq 1$), $I_6=K_a^c\vee (K_b\cup K_{c})$ ($a\geq 1$, $b,c\geq 2$) and $I_7=K_a^c\vee (K_b\cup K_{c}^c)$ ($a,c\geq 1$, $b\geq 2$); or
\vspace{-0.2cm}

\item $d(G)=3$ and \begin{enumerate}[1)]
\vspace{-0.2cm}

\item $G=J_1=P_4[K_{a}^c,K_{b},K_{c}^c,K_{d}^c]$, where $b\geq 1$, and $a=c=1$, $d\geq 1$ or $a=1$, $c=2$, $d\leq 2$ or $a=1$, $c\geq 3$, $d=1$ or $a=2$, $c=1$, $d\leq 2$ or $a\geq 3$, $c=1$, $d=1$; or
\vspace{-0.2cm}

\item $G=J_2=P_4[K_{a}^c,K_{b},K_{c},K_{d}^c]$, where $b,c\geq 1$, and $a=1$, $d\geq 1$ or $a=2$, $d\leq 2$ or $a\geq 3$, $d=1$; or
\vspace{-0.2cm}

\item $G=J_3=P_4[K_{a}^c,K_{b},K_{c}^c,K_{d}]$, where $b,d\geq 1$, and $a=1$, $c\geq 1$ or $a\geq 2$, $c=1$; or
\vspace{-0.2cm}

\item $G=J_4=P_4[K_{a}^c,K_{b}^c,K_{c},K_{d}]$, where $c,d\geq 1$, and $a=1$, $b\geq 1$ or $a=2$, $b\leq 2$ or $a\geq 3$, $b=1$; or
\vspace{-0.2cm}

\item $G=J_5=P_4[K_{a}^c,K_{b},K_{c},K_{d}]$, where $a,b,c,d\geq 1$; or
\vspace{-0.2cm}

\item $G=J_6=P_4[K_{a},K_{b},K_{c}^c,K_{d}]$, where $a,b,c,d\geq 1$; or
\vspace{-0.2cm}

\item $G=J_7=P_4[K_{a},K_{b},K_{c},K_{d}]$, where $a + b + c + d - 3ac - 8ad - 3bd - abc - abd - acd - bcd + abcd + 1\leq0$.
\end{enumerate}
\end{enumerate}
\end{thm}
\begin{proof}
According to Proposition \ref{Pro-3-1}, to determine the graphs with $\partial_3(G)\leq -1$ and $\partial_{n-1}(G)\geq -2$, it suffices to identify such graphs from $\{I_i,J_i\mid 1\leq i\leq 7\}$ by using  Proposition \ref{Pro-3-2}. Here we only check the graphs $I_7$, $J_1$ and $J_7$, and the remaining graphs could be checked in a similar way and so the detail is omitted.

First suppose $G=I_7=K_a^c\vee (K_b\cup K_{c}^c)$ ($a,c\geq 1$, $b\geq 2$). Then the $D$-polynomial of $G$ is equal to $\Phi_G(x)=(x+1)^{b-1}(x+2)^{a+c-2}\Psi_G(x)$ (see Tab. \ref{Table-3}), where $\Psi_G(x)=x^3 + (5 - b - 2c - 2a)x^2 + (ab - 4b - 6c - 6a + 3ac- 2bc + 8)x - 4a - 4b - 4c+ 2ab + 3ac - 4bc + 3abc + 4$. Let $\alpha_1>\alpha_2\geq\alpha_3$ be the three zeros of $\Psi_G(x)$. Note that  $\alpha_1=\partial_1(G)>0$ by Lemma \ref{Lemma-2-8}.  Since $\Psi_G(-1)=ab - b - 2bc + 3abc>0$  and $\Psi_G(-2)=3abc - 3ac>0$, we have $\alpha_1>\alpha_2>-1$ and $\alpha_3<-2$, which implies that $\partial_3(G)\leq -1$ and $\partial_{n-1}(G)\geq -2$.

Next suppose $G=J_1=P_4[K_{a}^c,K_{b},K_{c}^c,K_{d}^c]$, where $a,b,c,d\geq 1$. Then the $D$-polynomial of $G$ is $\Phi_G(x)=(x+1)^{b-1}(x+2)^{a+c+d-3}\Psi_G(x)$ (see Tab. \ref{Table-3}), where $\Psi_G(x)=x^4 + (7 - b - 2c - 2d - 2a)x^3 + (ab - 6b - 10c- 10d- 10a - 5ad + bc - 2bd+ 3cd + 18)x^2+ (4ab - 12b - 16c - 16d - 16a- 15ad+ 4bc - 8bd + 9cd + 3abd + 8acd + 3bcd + 20)x- 8a - 8b - 8c - 8d+ 4ab- 10ad+ 4bc - 8bd + 6cd + 6abd + 8acd + 6bcd - 4abcd + 8$. Let $\alpha_1>\alpha_2\geq\alpha_3\geq\alpha_4$ be the four zeros of $\Psi_G(x)$. Note that  $\alpha_1=\partial_1(G)>0$, and  $\alpha_4=\partial_n(G)\leq -3$ by Lemma \ref{Lemma-2-1}.  Also note that $\alpha_2=\partial_2(G)\geq 1-\sqrt{3}>-1$ by Lemma \ref{Lemma-2-6} since $G$ is not compete.  By simple computation, we get $\Psi_G(-2)=- 8acd - 4abcd<0$, which implies that $\alpha_3>-2$ since we have obtained  $\alpha_1>\alpha_2>-1$ and $\alpha_4\leq -3$, and so $\partial_{n-1}(G)\geq -2$. Furthermore, we see that $\partial_{3}(G)\leq -1$ if and only if $\alpha_3\leq -1$, which is the case if and only if $\Psi_G(-1)=ab - b + bc - 2bd + 3abd + 3bcd - 4abcd\geq0$ by above arguments. Now it suffices to determine those $a,b,c,d$ such that $\partial_{3}(G)\leq -1$. If $a,c\geq 2$, then $D(P_4[2K_1,K_1,2K_1,K_1])$ is a principal submatrix of $D(G)$, which implies that $-0.8990=\partial_{3}(P_4[2K_1,K_1,2K_1,K_1])\leq \partial_3(G)$ by Lemma \ref{Lemma-2-4}, a contradiction. Then we can suppose that $a=1$ or $c=1$. If $a=1$, then $\Psi_G(-1)=bc + bd - bcd$, and so $\Psi_G(-1)\geq0$ if and only if $b\geq 1$, and $c=1$, $d\geq 1$ or $c=2$, $d\leq 2$ or $c\geq 3$, $d=1$ by simple compution. Similarly, if $c=1$, then $\Psi_G(-1)=ab + bd - abd\geq 0$ if and only if $b\geq 1$, and $a=1$, $d\geq 1$ or $a=2$, $d\leq 2$ or $a\geq 3$, $d=1$. Combining above results, if $G=J_1=P_4[K_{a}^c,K_{b},K_{c}^c,K_{d}^c]$, then $\partial_3(G)\leq -1$ and $\partial_{n-1}(G)\geq -2$ if and only if $b\geq 1$, and $a=c=1$, $d\geq 1$ or $a=1$, $c=2$, $d\leq 2$ or $a=1$, $c\geq 3$, $d=1$ or $a=2$, $c=1$, $d\leq 2$ or $a\geq 3$, $c=1$, $d=1$.

Finally, we  suppose  $G=J_7=P_4[K_{a},K_{b},K_{c},K_{d}]$, where $a,b,c,d\geq 1$. Then $\Phi_G(x)=(x+1)^{a+b+c+d-4}\Psi_G(x)$, where $\Psi_G(x)=x^4 + (4 - b - c - d - a)x^3 + (6 - 3b - 3c - 3d- 3ac- 8ad- 3bd - 3a)x^2+ (abc - 3b - 3c - 3d - 6ac - 16ad - 6bd - 3a+ abd + acd + bcd + 4)x- a - b - c - d- 3ac - 8ad - 3bd + abc+ abd + acd + bcd + abcd + 1$. Let $\alpha_1>\alpha_2\geq\alpha_3\geq\alpha_4$ be the four zeros of $\Psi_G(x)$. As above, we have $\alpha_1=\partial_1(G)>0$, $\alpha_2=\partial_2(G)>-1$ and $\alpha_4=\partial_n(G)\leq -3$. By simple computation, we have $\Psi_G(-1)=abcd>0$, which implies that $\alpha_3< -1$, and so $\partial_{3}(G)\leq -1$. Moreover, we see that $\partial_{n-1}(G)\geq -2$ if and only if $\alpha_3\geq -2$, which is the case if and only if $\Psi_G(-2)=a + b + c + d - 3ac - 8ad - 3bd - abc - abd - acd - bcd + abcd + 1\leq0$ by above arguments. Therefore, we have $\partial_3(G)\leq -1$ and $\partial_{n-1}(G)\geq -2$ if and only if $a + b + c + d - 3ac - 8ad - 3bd - abc - abd - acd - bcd + abcd + 1\leq0$.
\end{proof}

\begin{remark}\label{rem-1}
\emph{
To investigate whether the graphs with $\partial_3(G)\leq -1$ and $\partial_{n-1}(G)\geq -2$ are DDS, it remains to compare the $D$-polynomials of $I_1$--$I_7$ and $J_1$--$J_7$ according to Theorem \ref{Thm-3-1}. The process of computation is complicated and tedious, so we do not discuss the DDS-property of these graphs in this paper. Indeed, there exist some non-isomorphic $D$-cospectral graphs in this class. For example, one can verify that $J_7^1=P_4[K_1,K_1,K_3,K_9]$ and  $J_7^2=P_4[K_1,K_9,K_1,K_3]$ are a pair of non-isomorphic $D$-cospectral graphs belonging to this class.
}
\end{remark}

\section{Graphs with at most three $D$-eigenvalues different from $-1$ and $-2$}\label{s-4}
For a connected graph  $G$ on $n$ vertices, we denote by $m_G(\partial)$  the multiplicity of $\partial$ as a $D$-eigenvalue of $G$.
 In this section, we focus on characterizing the graphs with at most three $D$-eigenvalues different from $-1$ and $-2$, that is, the graphs with $m_G(-1)+m_G(-2)\geq n-3$, which gives new families of graphs with few distinct $D$-eigenvalues. Clearly, we have $m_G(-1)+m_G(-2)\leq n-1$.  If $m_G(-1)+m_G(-2)=n-1$, then $\partial_2(G)\leq -1<1-\sqrt{3}$, implying that $G$ is the complete graph $K_n$ by Lemma \ref{Lemma-2-6}. Thus it suffices to determine those graphs with $m_G(-1)+m_G(-2)\in\{n-2,n-3\}$.

\begin{thm}\label{Thm-4-1}
Let $G$ a connected graph on $n\geq 4$ vertices. Then $m_G(-1)+m_G(-2)=n-2$ if and only if $G=K_{s,n-s}$ ($1\leq s\leq n-1$) or $G=K_{s}^c\vee K_{n-s}$ ($2\leq s\leq n-2$).
\end{thm}
\begin{proof}
Clearly, $G$ is not a complete graph due to $m_G(-1)<n-1$.  We consider the following three cases.

\vspace{0.2cm}
\noindent\textbf{Case 1.} $m_G(-1)=n-2$ and $m_G(-2)=0$.

By Lemma \ref{Lemma-2-1},  we have $\partial_n(G)\leq -2$  because $d(G)\geq 2$.  This implies that $\partial_2(G)=-1$ because $\partial_1(G)>0$ and $m_G(-1)=n-2>0$, and thus $G$ is a complete graph by Lemma \ref{Lemma-2-6}, which is a contradiction.

\vspace{0.2cm}
\noindent\textbf{Case 2.} $m_G(-1)=0$ and $m_G(-2)=n-2$.

In this situation, we can suppose that  $\mathrm{Spec}_D(G)=\{\alpha,\beta,[-2]^{n-2}\}$ with $\alpha>\beta>-2$ or $\mathrm{Spec}_D(G)=\{\alpha,  [-2]^{n-2},\beta\}$ with $\alpha>-2>\beta$. For the former, we have $\partial_{n}(G)=-2$ and so $G$ is a complete bipartite graph $K_{s,n-s}$ ($1\leq s\leq n-1$) according to Lemma \ref{Lemma-2-2}. Conversely, it is easy to verify that $-1$ is not a $D$-eigenvalue of $K_{s,n-s}$ due to $n\geq 4$. For the later, we have $\partial_{2}(G)=-2<1-\sqrt{3}$, and so $G$ is a complete graph, which is impossible. 

\vspace{0.2cm}
\noindent\textbf{Case 3.} $m_G(-1)\geq 1$, $m_G(-2)\geq 1$ and $m_G(-1)+m_G(-2)=n-2$.

In this situation, the $D$-spectrum of $G$ has three possible forms, i.e., $\mathrm{Spec}_D(G)=\{\alpha,\beta,[-1]^{m_1},[-2]^{m_2}\}$ with $\alpha>\beta>-1$, $\mathrm{Spec}_D(G)=\{\alpha,[-1]^{m_1},\beta,$ $[-2]^{m_2}\}$ with $\alpha>-1>\beta>-2$ or $\mathrm{Spec}_D(G)=\{\alpha,[-1]^{m_1},[-2]^{m_2},\beta\}$ with $\alpha>-1>-2>\beta$, where $m_1=m_G(-1)\geq 1$, $m_2=m_G(-2)\geq 1$ and $m_1+m_2=n-2$. We claim that the last two forms cannot occur since otherwise we have $\partial_2(G)=-1$, which is impossible because $G$ cannot be a complete graph. For the first form, we have $\partial_{n}(G)=-2$, and so $G$ is a complete $(n-m_2)$-partite ($n-m_2\geq 3$) graph according to Lemma \ref{Lemma-2-2}. Moreover, we claim that $G$ cannot contain $K_{2,2,1}=F_6$ (see Fig. \ref{Figure-2}) as its induced subgraph by Lemma \ref{Lemma-3-1} since $\partial_3(G)=-1$. Thus we may conclude that $G=K_{s,1,\ldots,1}=K_{s}^c\vee K_{n-s}$, where $s=m_2+1\in [2,n-2]$ because we have known that $G$ is a complete $(n-m_2)$-partite  graph. Conversely,  as in Proposition \ref{Pro-3-2}, one can easily check that $\mathrm{Spec}_D(K_{s}^c\vee K_{n-s})=\{\alpha,\beta,[-1]^{n-s-1},[-2]^{s-1}\}$, where $\alpha,\beta$ are the two zeros of $x^2 - (n+s-3)x-s^2+sn-2(n-1)$ satisfying  $\alpha>\beta>-1$ due to $2\leq s\leq n-2$.

We complete the proof.
\end{proof}

\begin{thm}\label{Thm-4-2}
Let $G$ be a connected graph with $n\geq 5$ vertices. Then $m_G(-1)+m_G(-2)=n-3$ if and only if $G$ is one of the following graphs: 
$K_a\vee(K_b\cup K_c)$  where $a+b+c\geq 5$ and $b+c\geq 3$; $K_{a,b,c}$ where $a+b+c\geq 5$;  $(K_{a}^c\vee K_b^c)\vee K_{c}$ where $a,b,c\geq 2$; $I_4=K_a\vee (K_b\cup K_c^c)$ where $a,b,c\geq 2$; $I_6=K_a^c\vee (K_b\cup K_{c})$ where $a,b,c\geq 2$; $I_7=K_a^c\vee (K_b\cup K_{c}^c)$ where $a+c\geq 3$ and $b\geq 2$; $J_1=P_4[K_{a}^c,K_{b},K_{c}^c,K_{d}^c]$ where $b\geq 1$ and $a=1$, $c=2$, $d=2$ or $a=2$, $c=1$, $d=2$; $J_2=P_4[K_{a}^c,K_{b},K_{c},K_{d}^c]$ where $b,c\geq 1$ and $a=d=2$; $J_4=P_4[K_{a}^c,K_{b}^c,K_{c},K_{d}]$ where $c,d\geq 1$ and $a=b=2$; $J_7=P_4[K_{a},K_{b},K_{c},K_{d}]$ where $a+b+c+d\geq 5$ and $a + b + c + d - 3ac - 8ad - 3bd - abc - abd - acd - bcd + abcd + 1=0$.
\end{thm}
\begin{proof}
Clearly, $G$ is not a complete graph due to $m_G(-1)<n-1$.  We consider the following three cases.

\vspace{0.2cm}
\noindent\textbf{Case 1.} $m_G(-1)=n-3$ and $m_G(-2)=0$.

Since $\partial_n(G)\leq -2$ and $\partial_1(G)>0$, we can suppose that $\mathrm{Spec}_D(G)=\{\alpha,[-1]^{n-3},$ $\beta,\gamma\}$ with $\alpha>-1>\beta\geq\gamma$ or $\mathrm{Spec}_D(G)=\{\alpha,\beta,[-1]^{n-3},\gamma\}$ with $\alpha>\beta>-1>\gamma$. Note that $G$ is not a complete graph. The former case cannot occur, and the later case implies that $\partial_{n-1}(G)=-1$ and so $G=K_a\vee(K_b\cup K_c)$ ($a,b,c\geq 1$ and $a+b+c=n\geq 5$) by Lemma \ref{Lemma-2-3}. Conversely, it is easy to check that $-1$ is a $D$-eigenvalue of  $K_a\vee(K_b\cup K_c)$ with multiplicity $n-3$, and $-2$ is a $D$-eigenvalue of $K_a\vee(K_b \cup K_c)$ if and only if $b=c=1$. Therefore, in this situation, we obtain that $G=K_a\vee(K_b\cup K_c)$, where $a+b+c=n\geq 5$ and $b+c\geq 3$.

\vspace{0.2cm}
\noindent\textbf{Case 2.} $m_G(-1)=0$ and $m_G(-2)=n-3$.

By Lemma \ref{Lemma-2-6}, we see that $-2$ cannot be the second largest $D$-eigenvalue of $G$. Thus it suffices to consider the following two situations.

\vspace{0.2cm}
\textbf{Subcase 2.1.} $\mathrm{Spec}_D(G)=\{\alpha,\beta,\gamma,[-2]^{n-3}\}$, where $\alpha>\beta\geq\gamma>-2$.

Since $\partial_{n}(G)=-2$ with multiplicity $n-3$, from Lemma \ref{Lemma-2-2} we have $G=K_{a,b,c}$, where $a+b+c=n\geq 5$. Also, it is easy to check that $-1$ cannot be a $D$-eigenvalue of $K_{a,b,c}$ because $a+b+c>3$, and so our result follows.

\vspace{0.2cm}
\textbf{Subcase 2.2.} $\mathrm{Spec}_D(G)=\{\alpha,\beta,[-2]^{n-3},\gamma\}$, where $\alpha>\beta>-2>\gamma$.

In this situation, we have $\partial_3(G)=-2$. First we claim that $G$ contains no induced $P_4$. If not, let $P_4=v_1v_2v_3v_4$ be an induced subgraph of $G$.  Then $2\leq d_G(v_1,v_4)\leq 3$. If $d_G(v_1,v_4)=3$, then $D(P_4)$ is a principal submatrix of $D(G)$, and so  $-1.1623=\partial_3(P_4)\leq \partial_3(G)=-2$ by Lemma \ref{Lemma-2-4}, a contradiction. If $d_G(v_1,v_4)=2$, then one of $\{F_1,F_2,F_3\}$ is the induced subgraph of $G$ (see Fig. \ref{Figure-2}), which is impossible because $\partial_3(F_i)>-2$ for $i=1,2,3$. Thus $G$ contains no induced $P_4$, and we can suppose $G=G_1\vee G_2$ by Lemma \ref{Lemma-2-10}, where $G_1$ and $G_2$ are non-null. Moreover, we conclude that both $G_1$ and $G_2$ contain no edges since $G$ contains no induced $K_3$ due to $\partial_3(K_3)=-1>-2=\partial_3(G)$. Then $G$ is a complete bipartite graph, and so $\partial_n(G)=-2$ by Lemma \ref{Lemma-2-2}, which contradicts $\partial_n(G)=\gamma<-2$. Therefore, there are no graphs satisfying $\mathrm{Spec}_D(G)=\{\alpha,\beta,[-2]^{n-3},\gamma\}$, where $\alpha>\beta>-2>\gamma$.

\vspace{0.2cm}
\noindent\textbf{Case 3.} $m_G(-1)\geq 1$, $m_G(-2)\geq 1$ and $m_G(-1)+m_G(-2)=n-3$.

By Lemma \ref{Lemma-2-6}  we know that $\partial_2(G)\neq -1$ because $G$ is not complete. Thus we only need to consider the following three cases.

\vspace{0.2cm}
\textbf{Subcase 3.1.} $\mathrm{Spec}_D(G)=\{\alpha,\beta,\gamma,[-1]^{m_1},[-2]^{m_2}\}$, where $\alpha>\beta\geq\gamma>-1$ and $m_1,m_2\geq 1$.

Since $\partial_{n}(G)=-2$,  from Lemma \ref{Lemma-2-2} we obtain that $G$ is a complete $(n-m_2)$-partite ($n-m_2\geq 4$) graph. Furthermore, we claim that $G$ cannot contain $K_{2,2,2,1}$ as its induced subgraph  since otherwise we have $-0.8730=\partial_4(K_{2,2,2,1})\leq \partial_4(G)=-1$ by Lemma \ref{Lemma-2-5}, which is a contradiction. Thus we may conclude that $G=K_{a,b,1,\ldots,1}=(K_{a}^c\vee K_b^c)\vee K_{c}$, where $a+b=m_2+2\in [3,n-2]$ and $c=n-a-b\in[2,n-3]$ because we have known that $G$ is a complete $(n-m_2)$-partite  graph. By simple computaion, we obtain $\Phi_G(x)=(x+2)^{a+b-2}(x+1)^{c-1}\Psi_G(x)$, where $\Psi_G(x)=x^3 + (5 - 2b - c - 2a)x^2 + (3ab - 6b - 4c - 6a + ac + bc + 8)x - 4a - 4b - 4c + 3ab + 2ac + 2bc - abc + 4$. Let $\alpha_1,\alpha_2,\alpha_3$ be the three zeros of $\Psi_G(x)$. Then $\alpha_1>0$ and $\alpha_2>-1$ because $G$ is not complete. Also note that $\Psi_G(-2)=-3ab-abc<0$ and $\Psi_G(-1)=-(a-1)(b-1)c\leq 0$. Then we have $\alpha_3>-1$ if and only if $a,b\geq 2$. Therefore, in this situation, we obtain that $G=(K_{a}^c\vee K_b^c)\vee K_{c}$, where $a,b,c\geq 2$.

\vspace{0.2cm}
\textbf{Subcase 3.2.} $\mathrm{Spec}_D(G)=\{\alpha,\beta,[-1]^{m_1},\gamma,[-2]^{m_2}\}$, where $\alpha>\beta>-1>\gamma>-2$ and $m_1,m_2\geq 1$.

In this situation, $G$ is a complete $(n-m_2)$-partite ($n-m_2\geq 4$) graph because $\partial_n(G)=-2$ is of multiplicity $m_2$. Also, as in Case 3 of the proof of Theorem \ref{Thm-4-1}, $K_{2,1,1}=F_6$ cannot be the induced subgraph of $G$ due to $\partial_3(G)=-1$. This also implies that $G$ is of the form $G=K_{s,1,\ldots,1}=K_{s}^c\vee K_{n-s}$, where $s=m_2+1\in [2,n-3]$. However, we have known that $\mathrm{Spec}_D(K_{s}^c\vee K_{n-s})=\{\alpha,\beta,[-1]^{n-s-1},[-2]^{s-1}\}$, contrary to $m_1+m_2=n-3$. Thus there are no graphs in this situation.

\vspace{0.2cm}
\textbf{Subcase 3.3.} $\mathrm{Spec}_D(G)=\{\alpha,\beta,[-1]^{m_1},[-2]^{m_2},\gamma\}$, where $\alpha>\beta>-1>-2>\gamma$ and $m_1,m_2\geq 1$.

In this situation, we have $\partial_3(G)=-1$ and $\partial_{n-1}(G)=-2$. Then $G$ is one of the graphs listed in Theorem \ref{Thm-3-1}. Therefore, it suffices to select from Theorem \ref{Thm-3-1} those graphs whose $D$-spectrum is of the from $\mathrm{Spec}_D(G)=\{\alpha,\beta,[-1]^{m_1},[-2]^{m_2},\gamma\}$, where $\alpha>\beta>-1>-2>\gamma$ and $m_1,m_2\geq 1$. With the help of Proposition \ref{Pro-3-2}, one can easily check that all the required graphs are: $I_4=K_a\vee (K_b\cup K_c^c)$ with $a,b,c\geq 2$; $I_6=K_a^c\vee (K_b\cup K_{c})$ with $a,b,c\geq 2$; $I_7=K_a^c\vee (K_b\cup K_{c}^c)$ with $a+c\geq 3$ and $b\geq 2$; $J_1=P_4[K_{a}^c,K_{b},K_{c}^c,K_{d}^c]$ with $b\geq 1$ and $a=1$, $c=2$, $d=2$ or $a=2$, $c=1$, $d=2$; $J_2=P_4[K_{a}^c,K_{b},K_{c},K_{d}^c]$ with $b,c\geq 1$ and $a=d=2$; $J_4=P_4[K_{a}^c,K_{b}^c,K_{c},K_{d}]$ with $c,d\geq 1$ and $a=b=2$; $J_7=P_4[K_{a},K_{b},K_{c},K_{d}]$ with $a+b+c+d\geq 5$ and $a + b + c + d - 3ac - 8ad - 3bd - abc - abd - acd - bcd + abcd + 1=0$.

We complete the proof.
\end{proof}


\begin{thebibliography}{}
\small{

\bibitem{Alazemi} A. Alazemi, M. Andeli\'{c}, T. Koledin, Z. Stani\'{c}, Distance-regular graphs with small number of distinct distance eigenvalues, Linear Algebra Appl. 531  (2017) 83--97.

\bibitem{Cheng} X.M. Cheng, A.L. Gavrilyuk, G.R.W. Greaves,  J.H. Koolen, Biregular graphs with three eigenvalues, Europ. J. Combin. 56 (2016) 57--80.

\bibitem{Cioaba} S.M. Cioab\u{a}, W.H. Haemers, J.R. Vermette, The graphs with all but two eigenvalues equal to $-2$ or $0$, Des. Codes Cryptogr.  84(1--2) (2017) 153--163.

\bibitem{Cioaba1} S.M. Cioab\u{a}, W.H. Haemers, J.R. Vermette, W. Wong, The graphs with all but two eigenvalues equal to $\pm1$, J. Algebraic Combin. 41(3) (2015) 887--897.


\bibitem{Godsil} C.D. Godsil, G. Royle, Algebraic Graph Theory, in: Graduate Texts in Mathematics, vol. 207, Springer, New York, 2001.

\bibitem{Huang}  X.Y. Huang, Q.X. Huang, On regular graphs with four distinct eigenvalues, Linear Algebra Appl. 512 (2017) 219--233. 

\bibitem{Jin} Y. Jin, X. Zhang, Complete multipartite graphs are determined by their distance spectra, Linear Algebra Appl. 448 (2014) 285--291.

\bibitem{Koolen} J.H. Koolen, S. Hayat, Q. Iqbal, Hypercubes are determined by their distance spectra, Linear Algebra Appl. 505 (2016) 97--108.

\bibitem{Li} D. Li, J. Meng, The graphs with the least distance eigenvalue at least $-\frac{1+\sqrt{17}}{2}$, Linear Algebra Appl. 493 (2016) 358--380.

\bibitem{Lin} H. Lin, On the least distance eigenvalue and its applications on the distance spread, Discrete math. 338 (2015) 868--874.

\bibitem{Lin2} H. Lin, Y. Hong, J. Wang, J. Shu, On the distance spectrum of graphs, Linear Algebra Appl. 439 (2013) 1662--1669.

\bibitem{Lin1} H. Lin, M. Zhai,  S. Gong, On graphs with at least three distance eigenvalues less than $-1$, Linear Algebra Appl. 458 (2014) 548--558.

\bibitem{Liu} R. Liu, J. Xue, L. Guo, On the second largest distance eigenvalue of a graph, Linear Multilinear Algebra 65 (2017) 1011--1021.

\bibitem{Lu}  L. Lu, Q. Huang, X. Huang, The graphs with exactly two distance eigenvalues different from $-1$ and $-3$, J. Algebraic Combin. 45(2) (2017) 629--647.

\bibitem{Mohammadian} A. Mohammadian, B. Tayfeh-Rezaie, Graphs with four distinct Laplacian eigenvalues, J. Algebraic Combin. 34(4) (2011) 671--682. 

\bibitem{Rowlinson} P. Rowlinson, On graphs with just three distinct eigenvalues, Linear Algebra Appl. 507 (2016) 462--473.

\bibitem{Seinsche} D. Seinsche, On a property of the class of $n$-colorable graphs, J. Combin. Theory Ser. B 16 (1974) 191--193.

\bibitem{Xing} R. Xing, B. Zhou, On the second largest distance eigenvalue, Linear Multilinear Algebra 64 (2016) 1887--1898.

\bibitem{Yu} G. Yu, On the least distance eigenvalue of a graph, Linear Algebra Appl. 439 (2013) 2428--2433.


}
\end{thebibliography}
\end{document}